\documentclass[12pt]{amsart}

\usepackage{amssymb,mathrsfs,amsmath,amsthm,color,bm,mathtools,bbm,wasysym,cases,mathdots,mathtools}
\usepackage[pagebackref]{hyperref}
\usepackage[noadjust]{cite}

\usepackage{enumerate}
\usepackage[centering]{geometry}
\allowdisplaybreaks

\newtheorem{theorem}{Theorem}[section]
\newtheorem*{theorem*}{Theorem}

\newtheorem*{question*}{Question}

\newtheorem*{conjecture*}{Conjecture}
\newtheorem{lemma}[theorem]{Lemma}
\newtheorem{proposition}[theorem]{Proposition}
\newtheorem{corollary}[theorem]{Corollary}
\newtheorem{definition}[theorem]{Definition}
\theoremstyle{remark}
\newtheorem{remark}{Remark}
\numberwithin{equation}{section}

\DeclareMathOperator{\hdim}{\dim_H}

\newcommand{\fldim}{\dim_{l^1}}

\newcommand{\SL}{\operatorname{SL}}

\newcommand{\SO}{\operatorname{SO}}
\newcommand{\inj}{\mathrm{inj}}
\newcommand{\Ad}{\operatorname{Ad}}
\newcommand{\VWA}{\operatorname{VWA}}

\newcommand{\qaq}{\mathrm{\quad and\quad}}
\newcommand{\dif}{{\, \mathrm d}}

\newcommand{\Q}{\mathbb Q}
\newcommand{\N}{\mathbb N}
\newcommand{\I}{\mathbb I}
\newcommand{\R}{\mathbb R}
\newcommand{\Z}{\mathbb Z}
\newcommand{\cha}{\mathbbm{1}}
\newcommand{\lm}{\mathcal L}
\renewcommand{\hm}{\mathcal H}

\newcommand{\hc}{\mathcal H_\infty}

\newcommand{\cx}{\mathcal X}

\newcommand{\cq}{\mathcal Q}

\newcommand{\cd}{\mathcal D}
\newcommand{\cb}{\mathcal B}

\newcommand{\cf}{\mathcal F}

\newcommand{\cg}{\mathcal G}

\newcommand{\cs}{\mathcal S}

\newcommand{\fg}{\mathfrak{g}}

\newcommand{\scg}{\mathscr G}
\newcommand{\scgd}{\mathscr G_{\mathrm{dim}}}
\newcommand{\scgm}{\mathscr G_{\mathrm{mea}}}

\newcommand{\bb}{\mathbf{b}}

\newcommand{\bd}{\mathbf{d}}
\newcommand{\bx}{\mathbf{x}}

\newcommand{\bp}{\mathbf{p}}

\newcommand{\by}{\mathbf{y}}

\newcommand{\bzero}{\mathbf{0}}

\newcommand{\bxi}{\bm\xi}

\newcommand{\boa}{{\bm\omega}}
\newcommand{\bta}{{\bm\theta}}

\newcommand{\bmeta}{\bm\eta}
\newcommand{\ve}{\varepsilon}

\begin{document}
\title[Jarn\'ik-typetheorem for self-similar sets]{Jarn\'ik-type theorem for self-similar sets}
\author{Yubin He}

\address{Department of Mathematics, Shantou University, Shantou, Guangdong, 515063, China}

\email{ybhe@stu.edu.cn}

	\author{Lingmin Liao}

	\address{School of Mathematics and Statistics, Wuhan University, Wuhan, Hubei 430072, China}

	\email{lmliao@whu.edu.cn}


\begin{abstract}
Let $K\subset\R^d$ be a compact subset equipped with a $\delta$-Ahlfors regular measure $\mu$.
For any $\tau>1/d$ and any ``inhomogeneous" vector $\bta\in\R^d$, let $W_d(\psi_\tau,\bta)$ denote the set of $(\psi_\tau,\bta)$-well approximable numbers,
where $\psi_\tau(q)=q^{-\tau}$. Assuming a local estimate for the $\mu$-measure of the intersections of $K$ with the neighborhoods of ``rational" vectors $(\bp+\bta)/q$,
we establish a sharp upper bound for the Hausdorff dimension of $K\cap W_d(\psi_\tau,\bta)$,
together with some nontrivial lower bounds when $\tau$ is below a certain threshold.
One of the lower bounds becomes sharp in the one-dimensional homogeneous case ($d=1, \theta=0$) for a class of sufficiently thick self-similar sets $K$,
and moreover $K\cap W_1(\psi_\tau,0)$ has full $(\delta+\frac{2}{1+\tau}-1)$-Hausdorff measure. These results have several applications:
\begin{enumerate}
	\item the set of homogeneous very well approximable numbers has full Hausdorff dimension within strongly irreducible self-similar sets in $\R^d$, extending a recent result of Chen [arXiv: 2510.17096];
	\item the set of inhomogeneous very well approximable numbers has full Hausdorff dimension within sufficiently thick missing digits sets in $\R$, affirmatively answering a question posed by Yu [arXiv: 2101.05910].
\end{enumerate}
Our applications build on the seminal works of Yu [arXiv: 2101.05910] and B\'enard, He and Zhang [arXiv:2508.09076].

We also provide some non-trivial missing digits set $K\subset[0,1]^d$ whose intersection with $W_d(\psi_\tau,\bzero)$ has full $(\delta+\frac{1+d}{1+\tau}-d)$-Hausdorff measure.
\end{abstract}
\maketitle

\section{Introduction}
Let $d\ge 1$ and $\lm_d$ be the $d$-dimensional Lebesgue measure. We write $|\cdot|$ for the standard supremum norm in $\R^d$. For any $\bx\in\R^d$, let $\|\bx\|$ denote the distance from $\bx$ to the nearest integer vector in $\Z^d$.
Let  $\psi:\mathbb{N}\rightarrow\mathbb{R}_+$ be a monotonic function and $\bta\in\R^d$. The set of $(\psi,\bta)$-well approximable numbers is defined by
\[W_d(\psi,\bta):=\{\bx\in\R^d:\|q\bx-\bta\|<\psi(q)\text{ for infinitely many }q\in\Z\}.\]
 If $\psi(q)=q^{-\tau}$, then we write $\psi_\tau$ in place of $\psi$. Since $W_d(\psi,\bta)$ is invariant under integer translations, we may restrict our attention to $\bx\in[0,1]^d$.
 The metric properties of the set $[0,1]^d\cap W_d(\psi,\bta)$ have been extensively studied in
 Diophantine approximation.
 A fundamental result in this direction is Khintchine’s theorem, which provides
 a zero--one law for the Lebesgue measure of $[0,1]\cap W_1(\psi,0)$ in terms of
 the convergence or divergence of a natural volume sum.

 \medskip

 \noindent\textbf{Khintchine's theorem} (\cite{Khintmeasure}). \emph{Let $\psi:\N\to\R_+$ be a monotonic function. Then,
 \[
 \lm_1\big([0,1]\cap W_1(\psi,0)\big)=
 \begin{cases}
 	0 &  \sum_{q=1}^{\infty}\psi(q)<\infty,\\[1ex]
 	1 &  \sum_{q=1}^{\infty}\psi(q)=\infty.
 \end{cases}
 \]}

 While Khintchine theorem gives a complete description of the Lebesgue measure
 of $[0,1]\cap W_1(\psi,0)$, it does not capture the finer geometric structure of such
 set in the null case.
 In order to quantify the size of $[0,1]\cap  W_1(\psi,0)$ beyond Lebesgue measure,
 it is natural to consider its Hausdorff dimension or measure.

 A seminal result in this direction is Besicovitch--Jarn\'ik's theorem, which
 determines the Hausdorff dimension of the sets of well approximable numbers for monotonic functions.

 \noindent\textbf{Besicovitch--Jarn\'ik's theorem} (\cite{BesicovitchJLMSDiophantine} and \cite{Jarnikdimension}).
 \emph{Let $\psi:\N\to\R_+$ be a monotonic function. Then,
 \[
 \hdim \big([0,1]\cap W_1(\psi,0)\big)
 = \frac{2}{1+\tau(\psi)},
 \]
 where
 \[\tau(\psi):=\liminf_{q\to\infty}\frac{-\log\psi(q)}{q}\]
 is the lower order of $\psi$ at infinity.
}

 The above dimension result was further extended by Jarn\'ik \cite{Jarnikmeasure} to a statement about
 the Hausdorff measure of well approximable sets.
 More precisely, one can determine the $s$-dimensional Hausdorff measure of
 $[0,1]\cap W_1(\psi_\tau,0)$ for all $s>0$.
 In this paper, we will adopt a special case of Jarn\'ik’s theorem that suffices
 for our purposes.

 \noindent\textbf{Jarn\'ik's theorem} (\cite{Jarnikmeasure}). \emph{Let $\tau>1$. Then,
 	\[
 	\hm^s\big([0,1]\cap W_1(\psi_\tau,0)\big)=
 	\begin{cases}
 		0 & \displaystyle s>\frac{2}{1+\tau},\\[1ex]
 		\infty & \displaystyle s\le\frac{2}{1+\tau}.
 	\end{cases}
 	\]}

 	Jarn\'ik's theorem turns out to be very useful in describing the Hausdorff measure of the sets of numbers with a given irrationality exponent.

 	For a point $\bx\in\R^d$, its {\em irrationality exponent} $\tau(\bx)$ is defined by
 	\[
 	\tau(\bx)
 	:= \sup\big\{\tau>0 : \bx\in W_d(\psi_\tau,0)\big\}.
 	\]
 With this definition in hand, we define the set of numbers with irrationality
 exponent equal to a given value $\tau$ by
 \[
 E_d(\tau,\bzero)
 := \big\{\bx\in\R^d: \tau(\bx)=\tau\big\}.
 \]

 \noindent\textbf{A corollary of Jarn\'ik's theorem.} \emph{Let $\tau>1$. Then,
 	\[
 	\hm^s\big([0,1]\cap E_1(\tau,0)\big)=
 	\begin{cases}
 		0 & \displaystyle s>\frac{2}{1+\tau},\\[1ex]
 		\infty & \displaystyle s\le\frac{2}{1+\tau}.
 	\end{cases}
 	\]}

For a proof of this corollary, see \cite{BereDickVelaniexact}. These one-dimensional homogeneous results, which provide a foundation for understanding the metric properties of well-approximable sets, have also been extended to higher dimensions $d \ge 2$ and to the inhomogeneous setting $\boldsymbol{\theta} \neq \boldsymbol{0}$ \cite{SchmidtmeasureTAMS, AllenRaminhomogeneousIMRN,BugeaudinhomogeneousJAM, BoroshFraenhigherJarnik, BereDickVelaniexact}.

 	In recent years, increasing attention has been paid to the study of intersections of $W_d(\psi,\bta)$ with fractal sets.
 	This line of investigation was motivated by a question of Mahler \cite{Mahler}:

 	\noindent\textbf{Question} (Mahler, \cite{Mahler}).
 	\emph{How closely can irrational elements of Cantor's set be approximated by rational
 	numbers?}

Although Mahler himself did not pose a more specific question, this nonetheless inspired subsequent studies on the metric properties of the intersection of $W_d(\psi,\bta)$ with self-similar
sets, of which the Cantor set is a special case. Recent years have witnessed significant progress in this direction,
particularly with the establishment of Khintchine-type theorems for self-similar measures and Besicovitch--Jarn\'ik-type theorems for missing digits sets. In the following, we discuss the problem in two parts, focusing separately on
the self-similar measure and the Hausdorff dimension.

For the self-similar measure, Kleinbock, Lindenstrauss and Weiss \cite{KleinLindWeissSelect} established
that if $\mu$ is a friendly measure—which includes strongly irreducible self-similar measures—then
$\mu(W_d(\psi_\tau,\mathbf{0})) = 0$ for all $\tau > 1/d$. These results were later extended by Pollington and Velani \cite{PollingtonVelaniSelect} to arbitrary functions $\psi$
for which a certain $\mu$-volume sum converges. Building on this foundation, Einsiedler, Fishman and Shapira \cite{EinFishShafractalsGAFA}, as well as Simmons and Weiss \cite{SimmonsWeissfractalInvent},
demonstrated that $\mu(W_d(\varepsilon \psi_{1/d},\mathbf{0}))$ has full measure for any $\varepsilon > 0$,
with $\mu$ being either a self-similar measure or an ergodic measure.
A complete Khintchine-type theorem for a large class of self-similar measures was established by Khalil and Luethi \cite{KhalilLuethifractalInvent}, as well as by Datta and Jana \cite{DattaJanafractals}, using different techniques.
The case of all strongly irreducible self-similar measures was finally settled by B\'enard, He and Zhang \cite{BenHeZhangfractalJAMS, BenHeZhangfractalRd}:

\noindent\textbf{Theorem BHZ} (\cite{BenHeZhangfractalJAMS,BenHeZhangfractalRd}). \emph{Let $\mu$ be strongly irreducible self-similar measure. Let $\psi:\N\to\R_+$ be a monotonic function. Then,
	\[
	\mu \big(W_d(\psi,\bzero)\big)=
	\begin{cases}
		0 &  \sum_{q=1}^{\infty}\psi(q)^d<\infty,\\[1ex]
		1 &  \sum_{q=1}^{\infty}\psi(q)^d=\infty.
	\end{cases}
	\]}

In the inhomogeneous setting, Chow and Yu \cite{ChowYufractals} established that the same zero--one law holds for any $\boldsymbol{\theta} \in \mathbb{R}^d$, not just for $\boldsymbol{\theta} = \boldsymbol{0}$, provided that $\fldim \mu > d - \frac{d-1}{d+1}$. The precise definition of $\fldim \mu$ will be given in Section \ref{ss:missing digit}.

For the Hausdorff dimension, comparatively little is known, and the corresponding problems
appear to be at least as challenging as those encountered in the context of Khintchine-type
theorems for self-similar measures. Levesley, Salp and Velani \cite{LevSalpVelaniMahlerMathAnn} first established that, for any missing digits set $K$, the Hausdorff dimension of its
intersection with $W_1(\psi_\tau,0)$ is at least $(\hdim K)/(1+\tau)$. In fact, the Cantor-type subset constructed in their proof only guarantees that points are well-approximated by rational points lying within the missing digits set $K$ itself. This type of problem is also referred to as {\em intrinsic Diophantine approximation}, which focuses on approximating points in a fractal set using rational points contained within the set rather than arbitrary rational points in $\Q$. For more related research, see, for example \cite{DaviaudIntrinsic,FishSimIntrinsicPLMS,LiLiWuIntrinsicMPCPS,LiVelaniWangIntrinsic,TanWangWuIntrinsicMZ}.

Since the rational points contained within the missing digits set $K$ are far fewer than those in $\R$, the lower bound $(\hdim K)/(1+\tau)$ is far from optimal. Levesley, Salp and Velani \cite{LevSalpVelaniMahlerMathAnn} further posed the following conjecture, which has recently been fully resolved by Chen \cite{Chenfractal}.
Define the set of \emph{very well approximable} points by
\[
\VWA_d(\bta) := \bigcup_{\tau > 1/d} W_d(\psi_\tau, \bta).
\]

\noindent\textbf{Theorem C} (\cite{Chenfractal}). \emph{Let $K$ be a self-similar set on $\R$. Then,
	\[
	\hdim\big(K\cap\VWA_1(0)\big)=\hdim K.
	\]}

Let $K_{1/3}$ be the middle-third Cantor set. Bugeaud and Durand \cite{BugeaudDurandJEMS} conjectured that the Hausdorff dimension of $K_{1/3}\cap W_1(\psi_\tau,0)$ should be given by the following explicit formula.

\noindent\textbf{Conjecture BD} (\cite{BugeaudDurandJEMS}). \emph{For any $\tau>1$,
	\[
	\hdim\big(K_{1/3}\cap W_1(\psi_\tau, 0)\big)=\max\bigg\{\frac{\log 2}{\log 3}+\frac{2}{1+\tau}-1,\ \frac{\log 2/\log 3}{1+\tau}\bigg\}.
	\]}

	For any $\tau > 1$ lying below a certain threshold, Chen \cite{Chenfractal} provided a sharp upper bound and a comparatively weaker lower bound for the Hausdorff dimension of $K\cap W_1(\psi_\tau, 0)$ for all self-similar sets $K$, not only for the middle-third Cantor set.
Another major advance was made by Yu \cite{Yufractals}, who in fact completed this line of research several years before Chen \cite{Chenfractal} by showing the following theorem. His results apply to sufficiently thick missing digits sets, but do not cover the middle-third Cantor set.

 \noindent\textbf{Theorem Y} (\cite{Yufractals}). \emph{Let $K\subset\R$ be a missing digits set, and let $\mu$ be the restriction of $\hm^{\hdim K}$ to $K$.
 	 If $\fldim\mu>1/2$, then
 	\begin{equation}\label{eq:Yuvwa}
 		\hdim\big(K\cap\VWA_1(0)\big)=\hdim K.
 	\end{equation}
 	Suppose further that $\hdim K\fldim\mu>1/2$, then there exists a constant $\alpha>1$ such that for any $\tau\in(1,\alpha)$,
 	\begin{equation}\label{eq:Yuwell}
 			\hdim\big(K\cap W_1(\psi_\tau,0)\big)=\hdim K+\frac{2}{1+\tau}-1.
 	\end{equation}
 	}

Let $\mu_{1/3}$ be the restriction of $\hm^{\frac{\log 2}{\log 3}}$ to $K_{1/3}$. It is currently unclear whether $\fldim \mu_{1/3}>1/2$, and consequently \eqref{eq:Yuvwa} in Theorem~Y cannot be applied to  $K_{1/3}$. Although the result for the Hausdorff dimension of $K_{1/3}\cap \VWA_1(0)$ has already been established by Chen \cite{Chenfractal}, Yu's approach still has its merits, and with suitable refinements, it can be extended to the inhomogeneous setting (see Theorem \ref{t:example2} (1)). The second conclusion \eqref{eq:Yuwell} in Theorem Y does not hold for $K_{1/3}$, since $\fldim\mu_{1/3}\le \hdim K_{1/3}$ (see \cite[Lemma 1.4 (1) and \S 3.6]{Yufractals}), and hence $\hdim K_{1/3}\fldim\mu_{1/3}\le(\frac{\log 2}{\log 3})^2\approx 0.63093^2<1/2$. Later, Chow, Varj\'u and Yu \cite{ChowVarjuYufractals} showed that \eqref{eq:Yuvwa} and \eqref{eq:Yuwell} in Theorem Y can be respectively applied to 5-ary and 7-ary missing digits sets, each with exactly one missing digit.

The starting point of this paper is inspired by the aforementioned results on Hausdorff dimension. Our goal is to develop a general framework that not only encompasses these previous results
but also allows for their natural generalization to higher dimensions, to the inhomogeneous setting,
and to results concerning the corresponding Hausdorff measures.

Throughout, the symbols $\ll$ and $\gg$ will be used to indicate an inequality with an unspecified positive multiplicative constant. By $a\asymp b$ we mean $a\ll b$ and $b\ll a$. Let $|E|$ denote the diameter of a set $E$. Let $K\subset\R^d$ be a compact subset equipped with a non-atomic probability measure $\mu$.
Suppose there exists a constant $\delta > 0$ such that, for any $\bx \in K$ and any $0 < r < |K|$,
\[
\mu\bigl(B(\bx,r)\bigr) \asymp r^\delta,
\]
where the implied constants are independent of $\bx$ and $r$.
A measure satisfying this condition is said to be \emph{$\delta$-Ahlfors regular}.

For any $\bta\in\R^d$, $Q>1$ and $\eta>0$, define
\[\begin{split}
	A_Q(\eta,\bta):&=\{\bx\in\R^d:\|q\bx-\bta\|<\eta\text{ for some $Q\le q<2Q$}\}\\
	&=\bigg\{\bx\in\R^d:\bigg|\bx-\frac{\bp+\bta}{q}\bigg|<\frac{\eta}{q}\text{ for some $\frac{\bp}{q}\in\Q^d$ with $Q\le q<2Q$}\bigg\}.
\end{split}\]

We now introduce a local counting property which captures the
distribution of the sets $A_Q(\eta,\bta)$ at small scales
with respect to the measure $\mu$.

\begin{definition}\label{d:local counting property}
	Let $\alpha>1/d$, $\beta>0$ and $\bta\in\R^d$. We say that the measure $\mu$ satisfies the
	\emph{$(\alpha,\beta,\boldsymbol{\theta})$-local counting property}
	if for all sufficiently large $Q\ge 1$ (depending only on $\alpha$ and $\beta$), all $Q^{-\alpha}\le\eta\le Q^{-1/d}$,
	and all balls $B$ whose centres lie in the support of $\mu$ and satisfy
	$|B|\ge Q^{-\beta}$, one has
	\begin{equation}\label{eq:local counting}
	\mu\bigl(B\cap A_Q(\eta,\bta)\bigr)
	\asymp \mu(B)\cdot Q\eta^d,
\end{equation}
	where the implied constants are independent of $Q$, $\eta$ and $B$.
\end{definition}
\begin{remark}\label{r:alphabeta}
	The parameters $\alpha$ and $\beta$ in Definition~\ref{d:local counting property}
	are not independent.
	To see this, suppose for simplicity that $\mu$ is a $\delta$-Ahlfors regular measure.
Applying \eqref{eq:local counting} with $\eta=Q^{-\alpha}$, we in particular obtain that $B\cap A_Q(Q^{-\alpha},\boldsymbol{\theta})\neq\varnothing$. Hence,
	$B\cap A_Q(4Q^{-\alpha},\boldsymbol{\theta})$ contains a ball of radius
	$Q^{-1-\alpha}$.
	Since $\mu$ is $\delta$-Ahlfors regular, it follows that
	\[
	\mu\bigl(B\cap A_Q(4Q^{-\alpha},\boldsymbol{\theta})\bigr)
	\gg Q^{-\delta(1+\alpha)}.
	\]
	On the other hand, the $(\alpha,\beta,\boldsymbol{\theta})$-local counting property yields
	\[
	\mu\bigl(B\cap A_Q(4Q^{-\alpha},\boldsymbol{\theta})\bigr)
	\asymp \mu(B)\cdot Q^{1-\alpha d}.
	\]
	Combining these estimates and using the lower bound $|B|\ge Q^{-\beta}$,
	together with the $\delta$-Ahlfors regularity of $\mu$, we obtain the necessary condition
	\[
	Q^{-\beta\delta}\cdot Q^{1-\alpha d}\gg Q^{-\delta(1+\alpha)}\quad\Longrightarrow\quad \alpha\le \frac{1+\delta-\beta\delta}{d-\delta}.
	\]
\end{remark}

The following theorem provides a sharp upper bound and a reasonably good lower bound  for the Hausdorff dimension
of well-approximable points on a compact set $K\subset \R^d$ whose measure, supported on $K$, satisfies the $(\alpha,\beta,\bta)$-local counting property.
\begin{theorem}\label{t:main}
	Let $K\subset\R^d$ be a compact subset equipped with a $\delta$-Ahlfors regular probability measure $\mu$. Suppose that $\mu$ satisfies the $(\alpha,\beta,\boldsymbol{\theta})$-local counting property for some $1/d<\alpha<\frac{1+\delta-\beta\delta}{d-\delta}$, $0<\beta<1$ and $\bta\in\R^d$. Let $\psi$ be a monotonic approximation function. If $1/d<\tau(\psi)<\alpha$, then
	\begin{equation}\label{eq:upper bound}
		\hdim \big(K\cap W_d(\psi,\bta)\big)\le \delta+\frac{d+1}{1+\tau(\psi)}-d,
	\end{equation}
	and
	\begin{equation}\label{eq:lower bound}
\begin{split}
	&\hdim\big(K\cap W_d(\psi,\bta)\big)\\
	\ge& \min\biggl\{\delta-\frac{(\tau(\psi) d-1)(\delta+\frac{1+d}{1+\tau(\psi)}-d)}{\beta\delta},\ \delta+\frac{1+d}{1+\tau(\psi)}-d\biggr\}.
\end{split}
	\end{equation}
	Moreover, if $\delta>d-1$, $\bta=\bzero$ and $\psi(q)\ge cq^{-\tau}$ with $1/d<\tau<\min\{\alpha,\frac{\delta+2-d}{2d-\delta-1}\}$ for infinitely many $q$. Then,
	\begin{equation}\label{eq:lower measure}
		\hm^s\big(K\cap W_d(\psi,\bf 0)\big)=\infty,
	\end{equation}
	where
	\begin{equation}\label{eq:s}
		s=\min\bigg\{\delta-\frac{(\tau-\frac{1}{d})(2d-\delta-1)(\delta+\frac{1+d}{1+\tau}-d)}{\beta\delta},\ \delta+\frac{1+d}{1+\tau}-d\bigg\},
	\end{equation}
	is hence a lower bound for $\hdim (K\cap W_d(\psi,\bf 0))$.
\end{theorem}
\begin{remark}\label{r:reason}
	By Conjecture BD, the upper bound \eqref{eq:upper bound} is sharp.
	The two lower bounds described in \eqref{eq:lower bound} and \eqref{eq:s}, on the other hand, are relatively weaker.
In fact, when $\delta>d-1$, the lower bound described in \eqref{eq:s} is strictly larger than that in \eqref{eq:lower bound} since
\[(\tau-1/d)(2d-\delta-1)<(\tau-1/d)\cdot d=\tau d-1,\]
and in certain cases, this  bound, namely \eqref{eq:s}, is even sharp. Specifically, as we will see in Remark \ref{r:fail}, the minimum in \eqref{eq:s} is
$\delta + \frac{1+d}{1+\tau} - d$ provided that
	\[
\delta \ge 2d-1-\frac{d^2\beta}{1+d}\qaq \beta > 1-d^{-2}.
\]
It is easy to observe that the most restrictive condition here is on $\beta$. When $d=1$, we only require $\beta > 0$, which is relatively easy to satisfy. However, as $d$ increases, $\beta$ must also increase and approach $1$.
It is unclear, even for the Lebesgue measure, whether the corresponding $\beta$ can be taken sufficiently close to $1$ when $d>1$. A crucial reason why the lower bound is more effective in one dimension is the simple fact that the distance between any two distinct rationals $p/q$ and $p'/q'$ is at least $1/(qq')$.
Although there are analogous results (see Lemma \ref{l:simplex lemma}) in higher dimensions, they are not strong enough, so it is necessary to require $\beta > 1 - d^{-2}$.
\end{remark}
\begin{remark}
	It is easy to see that, in Theorem \ref{t:main}, the larger the value of $\beta$, the better the lower bounds.
	However, in general, $\beta$ tends to be quite small.
	If one hopes to apply this result to Conjecture BD, it would require $\beta>2(1-\frac{\log 2}{\log 3})\approx0.7381$.
	According to the currently known results, achieving such a large value of $\beta$ seems to be very difficult, see also Remark \ref{r:self-similar}.
\end{remark}

\begin{remark}\label{r:beta arise}
	The appearance of $\beta$ in the two lower bounds \eqref{eq:lower bound} and \eqref{eq:s} stems from the following reason.
	When analyzing the fractal structure of the set $K\cap W_d(\psi,\bta)$, the $(\alpha,\beta,\bta)$-local counting property allows us to obtain good estimates for balls with radius larger than $Q^{-\beta}$.
	However, for balls with smaller radius, such estimates are no longer available,
	and one must resort to worst-case estimates.
	This limitation is what can cause the lower bound to be smaller than the value one might ideally expect.
\end{remark}


Next, we present several consequences, which, as we shall see later, extend the works of Chen \cite{Chenfractal} and Yu \cite{Yufractals}. Note that when $d=1$, the additional restriction on $\tau$ (imposed so as to obtain the better lower bound \eqref{eq:s}) is
\[
1<\tau<\min\left\{\alpha,\ \frac{1+\delta}{1-\delta}\right\}=\alpha,
\]
since  $\alpha<\frac{1+\delta-\beta\delta}{1-\delta}$.

\begin{corollary}\label{c:corollary}
	Let $K\subset\R^d$ be a compact subset equipped with a $\delta$-Ahlfors regular probability measure $\mu$. Suppose that $\mu$ satisfies the $(\alpha,\beta,\boldsymbol{\theta})$-local counting property for some $1/d<\alpha<\frac{1+\delta-\beta\delta}{d-\delta}$, $0<\beta<1$ and $\bta\in\R^d$. Then, the following statements hold.
	\begin{enumerate}
		\item $\hdim\big(K\cap\VWA_d(\bta)\big)=\hdim K=\delta.$
		\item If $d=1$, $\theta=0$, $\beta>0$ and $\delta>1-\beta/2$,
		then for any $1<\tau<\alpha$,
		\[\hm^{\delta+\frac{2}{1+\tau}-1}\big(K\cap E_1(\tau,0)\big)=\infty.\]
	\end{enumerate}
\end{corollary}

Based on this corollary, we provide several concrete examples.

\begin{theorem}\label{t:example1}
	Let $K\subset \R^d$ be a strongly irreducible self-similar set satisfying the open set condition. Suppose that $\bta=\bzero$. Then, the following statements hold.
	\begin{enumerate}
		\item For any $d\ge 1$, we have $\hdim (K\cap \VWA_d(\bzero))=\hdim K$.
		\item Let $d=1$. There exist computable constants $0<\delta_0<1$ and $\alpha>1$ such that for any $K$ with $\hdim K\ge\delta_0$, we have
		\[\hm^{\delta+\frac{2}{1+\tau}-1}\big(K\cap E_1(\tau,0)\big)=\infty,\quad\text{for any $\tau\in(1,\alpha)$.}\]
	\end{enumerate}
\end{theorem}
\begin{remark}
	Theorem \ref{t:example1} (1) extends Theorem C to higher dimensions.
	Moreover, for sufficiently thick self-similar sets, we are able to obtain the Hausdorff measure of the set $K \cap E_1(\tau,0)$,
	which generalizes Theorem Y on missing digits sets to thick self-similar sets.
	However, Theorem \ref{t:example1} (2) cannot be straightforwardly extended to higher dimensions, as explained in Remarks \ref{r:reason} and \ref{r:fail}.
\end{remark}
\begin{remark}\label{r:self-similar}
	The constants $\delta_0$ and $\alpha$ in Theorem \ref{t:example1} (2) are, in principle, computable.
	They depend on the exponent $\kappa$ appearing in the error term of the effective equidistribution result established by B\'enard, He and Zhang \cite{BenHeZhangfractalJAMS,BenHeZhangfractalRd} (see Theorem \ref{t:effective equidistribution}).
Roughly speaking, the admissible range of the parameter $\beta$ depends on the exponent $\kappa$, and increases as $\kappa$ becomes larger. More precisely, by Corollary~\ref{c:self-similar set}, $\beta$ can be taken arbitrarily close to $\frac{2\kappa}{2+\kappa}$ from below.
	However, the authors of \cite{BenHeZhangfractalJAMS,BenHeZhangfractalRd} do not provide an explicit expression for $\kappa$, and in fact the available value of $\kappa$ is very small.
	In order to apply the result to the middle-third Cantor set, it follows from Corollary \ref{c:corollary} (2) that one would require $\delta=\frac{\log 2}{\log 3}>1-\beta/2>1-\frac{\kappa}{2+\kappa}$, or equivalently $\kappa>2(1-\frac{\log 2}{\log 3})/\frac{\log 2}{\log 3}\approx 1.1699$, which seems difficult to achieve using their method.
	For this reason, we do not attempt to compute the explicit values of $\delta_0$ and $\alpha$ here.
\end{remark}

\begin{theorem}\label{t:example2}
	Let $K\subset[0,1]$ be a missing digits set, and let $\mu$ be the restriction of $\hm^{\hdim K}$ to $K$. Then, the following statements hold.
				\begin{enumerate}
					\item  If $\fldim\mu>1/2$, then for any $\theta\in\R$,
					\[\hdim \big(K\cap \VWA_1(\theta)\big)=\hdim K.\]
					\item If $\hdim K\fldim\mu>1/2$,
					then there exists $\alpha>1$ such that
					\[\hm^{\delta+\frac{2}{1+\tau}-1}\big(K\cap E_1(\tau,0)\big)=\infty,\quad\text{for any $\tau\in(1,\alpha)$.}\]
				\end{enumerate}
\end{theorem}
\begin{remark}
Theorem \ref{t:example2} (1) provides an affirmative answer to a question posed by Yu \cite[Remark 9.4]{Yufractals} concerning inhomogeneous Diophantine approximation on missing digits sets.
Theorem \ref{t:example2} (2) goes further by strengthening the conclusion of Theorem~Y from Hausdorff dimension to Hausdorff measure.
As mentioned earlier, it follows from the work of Chow, Varj\'u, and Yu \cite[Propositions 2.4 and 2.5]{ChowVarjuYufractals} that Theorem \ref{t:example2} (1) and (2) apply respectively to 5-ary and 7-ary missing digits sets, each with exactly one missing digit.
\end{remark}

As noted in Remarks \ref{r:reason} and \ref{r:beta arise}, the $(\alpha,\beta,\bzero)$-local counting property alone is insufficient to determine the Hausdorff dimension of $K\cap W_d(\psi_\tau,\bzero)$ for $d>1$. It is therefore natural to ask whether there exists a non-trivial fractal set $K$ for which the Hausdorff dimension of $K\cap W_d(\psi_\tau,\bzero)$ is exactly $\delta+\frac{1+d}{1+\tau}-d$. The following result shows that certain non-trivial missing digits sets provide a positive answer to this question.
\begin{theorem}\label{t:example3}
	Let $K\subset[0,1]^d$ be a missing digits set satisfying the following assumptions:
	\begin{enumerate}[({A}1)]
		\item $K = K_1 \times \cdots \times K_{d-1} \times [0,1]$, where $K_j$ ($1\le j\le d-1$) is a missing digits set;
		\item there exists  $\frac{d}{d+1} < \gamma < 1$ such that, for each $1 \le j \le d-1$, the restriction of $\hm^{\hdim K_j}$ to $K_j$, denoted by $\mu_j$, satisfies $\fldim \mu_j > \gamma$.
	\end{enumerate}
	Then, there exists $\alpha>1/d$ such that
	\[\hm^{\delta+\frac{1+d}{1+\tau}-d}\big(K\cap E_d(\tau,\bzero)\big)=\infty,\quad\text{for any $\tau\in(1/d,\alpha)$.}\]
\end{theorem}
\begin{remark}\label{r:higher dimension}
	The key ingredient of Theorem \ref{t:example3} can be described as follows. It is important to note that the classical mass transference principle (MTP for short) of Beresnevich and Velani~\cite{BeresnevichVelaniMTPann} cannot be applied directly in this setting. The main obstacle is that some rational vectors $\bp/q$ may lie outside $K$. In particular, it can happen that $B(\bp/q,r) \cap K \neq \varnothing$ while $B(\bp/q,r^\tau) \cap K = \varnothing$ for any $\tau > 1$ unless $K=[0,1]^d$. A genuine way to bypass this difficulty is to shrink the ball not in every direction, but only in, for example, the $d$-th direction. This motivates the requirement that the $d$-th component of $K$ be the full interval $[0,1]$,
	so that shrinking a ball in this direction maintains a substantial intersection with $K$.  More precisely, let $K$ be a missing digits set satisfying (A1), and let $\prod_{j=1}^{d} B(\frac{p_j}{q}, r_j) \subset [0,1]^d$. If $K\cap \prod_{j=1}^{d}B(\frac{p_j}{q},r_j)\ne\varnothing$, then
	\[K\cap \left(\Bigg(\prod_{j=1}^{d-1}B\bigg(\frac{p_j}{q},r_j\bigg)\Bigg)\times B\bigg(\frac{p_d}{q},r_d^\tau\bigg)\right)\ne\varnothing\quad\text{for any $\tau>1$},\]
	since the $d$-th component of $K$ is $[0,1]$.
	This strategy allows us to adapt the idea of the MTP without the stronger assumption that $K = [0,1]^d$. However, because the original MTP applies only to $\limsup$ sets defined by balls, the approach we adopt here is closer in spirit to the MTP from rectangles to small open sets \cite[Theorem 2.11]{HeMTPadv}, as established by the first author.
\end{remark}

\subsection*{Organization of the paper}

In Section \ref{s:preliminaries}, we introduce a useful tool for proving Theorems \ref{t:main} and \ref{t:example3}.
Sections \ref{s:upper bound} and \ref{s:lower bound} are devoted to proving Theorem \ref{t:main}, while Section \ref{s:corollary} presents the proof of Corollary \ref{c:corollary}.
In Section~\ref{s:example}, we focus on two specific classes of sets: self-similar sets and missing digits sets, which are treated in two separate subsections. More precisely, we verify that these two classes satisfy the $(\alpha,\beta,\bta)$-local counting property for some parameters $\alpha$, $\beta$ and $\bta\in\R^d$, respectively, which in turn yields Theorems~\ref{t:example1} and~\ref{t:example2}. In Section \ref{s:higher dimension}, we prove Theorem \ref{t:example3}.

\section{Hausdorff measure and content}\label{s:preliminaries}
Here and hereafter, we will assume
that $K\subset\R^d$ is a compact subset equipped with a $\delta$-Ahlfors regular probability measure $\mu$.

In this section, we introduce a useful tool for estimating the Hausdorff dimension and measure of the set $K\cap W_d(\psi,\bta)$.

Let $0<s\le d$. For a set $E\subset \R^d$ and $\eta>0$, let
\[\mathcal H_\eta^s(E):=\inf\bigg\{\sum_{i}|B_i|^s:E\subset \bigcup_{i\ge 1}B_i, \text{ where $B_i$ are balls with $|B_i|\le \eta$}\bigg\}.\]
The {\em $s$-dimensional Hausdorff measure} of $E$ is defined as
\[\hm^s(E):=\lim_{\eta\to 0^+}\mathcal H_\eta^s(E).\]
When $\eta=\infty$, $\hc^s(E)$ is referred to as {\em $s$-dimensional Hausdorff content} of $E$.

In \cite{FalconerLIPintroduce}, Falconer introduced and systematically developed the notion of the large intersection property, which provides a powerful theoretical framework for studying the Hausdorff dimension of $\limsup$ sets. This property ensures that certain subsets maintain large Hausdorff dimension even under countable intersections, making it particularly useful in Diophantine approximation.

\begin{definition}[\cite{FalconerLIPintroduce}]\label{d:LIPdim}
	Let $0<s\le \hdim K$. We define $\scg_{\dim}^s(K)$ to be the class of $G_\delta$-subsets $E$ of $K$ such that there exists a constant $c>0$ satisfying the following property: for any $0<t<s$ and any ball $B \subset K $,
	\begin{equation}\label{eq:defhcbounddim}
		\hc^t(E\cap B)>c\hc^t(B).
	\end{equation}
\end{definition}
	Intuitively, the condition \eqref{eq:defhcbounddim} ensures that the set $E$ is ``large'' in every ball, in the sense that it occupies a uniformly positive proportion of the $t$-dimensional Hausdorff content of $B$ for all $t<s$.

\begin{theorem}[\cite{FalconerLIPintroduce}]\label{t:LIPdim}
	Let $0<s\le d$. The class $\scg_{\dim}^s(K)$ is closed under countable intersections. Moreover, for any $E\in\scg_{\dim}^s(K)$, we have
	\[
	\hdim E\ge s.
	\]
\end{theorem}

However, Falconer's original definition could only provide lower bounds for the Hausdorff dimension, not the Hausdorff measure. This limitation was recently overcome by the first author in \cite{HeMTPadv}, who introduced a refined version of Definition \ref{d:LIPdim}.

\begin{definition}[{\cite[Definition 2.3]{HeMTPadv}}]\label{d:LIPmea}
	Let $0<s\le \hdim K$. We define $\scgm^s(K)$ to be the class of $G_\delta$-subsets $E$ of $K$ such that there exists a constant $c>0$ satisfying the following property: for any ball $B \subset K $,
	\begin{equation}\label{eq:defhcboundmea}
		\hc^s(E\cap B)>c\hc^s(B).
	\end{equation}
\end{definition}
The difference lies in whether the critical value $t=s$ can be attained. If the critical value cannot be attained, it is therefore natural that one cannot obtain information about the Hausdorff measure, as the Hausdorff measure is particularly sensitive to the critical dimension.
On the other hand, if the critical value can be attained, one may expect to obtain information about the Hausdorff measure, as described in the theorem below.

\begin{theorem}[{\cite[Theorem 2.4]{HeMTPadv}}]
	Let $0<s\le d$. The class $\scgm^s(K)$ is closed under countable intersections. Moreover, for any $E\in\scgm^s(K)$, we have
	\[
	\hm^s(E)=\hm^s(K).
	\]
\end{theorem}

Inspired by the mass transference principle from balls to open sets \cite{KoivusaloRamsMTPIMRN,ZhongMTPJMAA}, the author \cite{HeMTPadv} derived conditions that are significantly weaker than the original definitions, while still yielding the corresponding large intersection property.

\begin{theorem}[{\cite[Corollary 2.6]{HeMTPadv}}]\label{t:weaken}
	Let $0<s\le d$. Assume that $\{B_k\}$ is a sequence of balls in $K$ with radii tending to 0, and that $\mu(\limsup B_k)=1$. Let $\{E_n\}$ be a sequence of open sets (not necessarily contained in $B_k$). The following statement hold.
	\begin{enumerate}[\upshape(1)]
		\item If for any $0<t<s$, there exists a constant $c_t>0$ such that for any $B_k$,
		\[\limsup_{n\to\infty}\hc^t (E_n\cap B_k)>c_t\mu(B_k),\]
			then $\limsup E_n\in\scgd^s(K)$.
			\item If there exists a constant $c_s>0$ such that for any $B_k$,
			\[\limsup_{n\to\infty}\hc^s (E_n\cap B_k)>c_s\mu(B_k),\]
			then $\limsup E_n\in\scgm^s(K)$.
	\end{enumerate}
\end{theorem}
With this result now at our disposal, the main ideas to prove Theorem \ref{t:main} are to verify certain sets under consideration satisfying some Hausdorff content bound. For this purpose, the following mass distribution principle will be crucial.
\begin{proposition}[Mass distribution principle {\cite[Lemma 1.2.8]{BishopPeresbook}}]\label{p:MDP}
	Let $ E $ be a Borel subset of $ \R^d $. If $ E $ supports a Borel probability measure $ \nu $ that satisfies
	\[\nu(B)\le c|B|^s,\]
	for some constant $ 0<c<\infty $ and for every ball $B$, then $ \hc^s(E)\ge1/c $.
\end{proposition}

\section{Proof of Theorem \ref{t:main}: upper bound of $\hdim\big(K\cap W_d(\psi,\theta)\big)$}\label{s:upper bound}

Before giving the proof of the upper bound for
$\hdim(K\cap W_d(\psi,\bta))$, we first investigate the geometric structure of the set
$K \cap A_Q(\eta,\boldsymbol{\theta})$.
In particular, we show that, under a counting property weaker than the $(\alpha,\beta,\boldsymbol{\theta})$-local counting property,
this set admits sharp covering estimates.

\begin{lemma}\label{l:global counting}
	Suppose that for all sufficiently large $Q\ge 1$ and all $Q^{-\alpha}\le\eta\le Q^{-1/d}$, one has
	\begin{equation}\label{eq:global counting}
		\mu\bigl(A_Q(\eta,\bta)\bigr)
		\asymp Q\eta^d,
	\end{equation}
	where the implied constants are independent of $Q$ and $\eta$. Then, for all sufficiently large $Q\ge 1$ and all $Q^{-\alpha}\le\eta\le Q^{-1/d}$, the set $K\cap A_Q(\eta,\bta)$ can be covered by
	\[
	\asymp Q^{1+\delta}\eta^{d-\delta}
	\]
	balls of radius $Q^{-1}\eta$.
\end{lemma}

\begin{proof}
	The conclusion does not follow directly from \eqref{eq:global counting}, since
	the `rational' points $(\bp+\bta)/q$ need not lie in $K$, and therefore the size
	of the intersection
	\[
	K\cap B\left(\frac{\bp+\bta}{q},\,\frac{\eta}{q}\right)
	\]
	may vary significantly for different choices of $(\bp+\bta)/q$.

	Let $\bx\in K\cap A_Q(\eta,\bta)$. By definition, there exists
	$(\bp+\bta)/q$ with $Q\le q<2Q$ such that
	\[
	\left|\bx-\frac{\bp+\bta}{q}\right|<\frac{\eta}{q}\le \frac{\eta}{Q}.
	\]
	If $\by$ satisfies $|\bx-\by|<Q^{-1}\eta$, then
	\[
	\left|\by-\frac{\bp+\bta}{q}\right|
	\le |\by-\bx|+\left|\bx-\frac{\bp+\bta}{q}\right|
	<\frac{2\eta}{Q}<\frac{4\eta}{q},
	\]
	which implies that $\by\in A_Q(4\eta,\bta)$. Consequently,
	\begin{equation}\label{eq:covering inclusion}
		K\cap A_Q(\eta,\bta)
		\subset \bigcup_{\bx\in K\cap A_Q(\eta,\bta)} B(\bx,Q^{-1}\eta)
		\subset K\cap A_Q(4\eta,\bta).
	\end{equation}
	By our assumption
	\eqref{eq:global counting}, the $\mu$-measures of the two outer sets in
	\eqref{eq:covering inclusion} are comparable. Therefore,
	\begin{equation}\label{eq:muunion}
		\mu\bigg(\bigcup_{\bx\in K\cap A_Q(\eta,\bta)} B(\bx,Q^{-1}\eta)\bigg)
		\asymp Q\eta^d.
	\end{equation}
	By the $\delta$-Ahlfors regularity of $\mu$, this implies that
	$K\cap A_Q(\eta,\bta)$ can be covered by
	\[
	\asymp \frac{Q\eta^d}{(Q^{-1}\eta)^\delta}
	= Q^{1+\delta}\eta^{d-\delta}
	\]
	balls of radius $Q^{-1}\eta$.
\end{proof}

Let $\alpha$ be as in Theorem~\ref{t:main}, and assume that $1/d < \tau(\psi)<\alpha$.
Obviously, the $(\alpha,\beta,\boldsymbol{\theta})$-local counting property implies the global estimate \eqref{eq:global counting}, and hence Lemma~\ref{l:global counting}
provides a sharp covering of $K \cap A_Q(\eta,\boldsymbol{\theta})$ with $Q^{-\alpha}\le\eta\le Q^{-1/d}$. Accordingly, to obtain the upper bound of $K\cap W_d(\psi,\bta)$, we define
\[\phi(q)=\max\{\psi(q),q^{-\alpha}\}.\]
It is evident that $K\cap W_d(\psi,\bta)\subset K\cap W_d(\phi,\bta)$ and  $\phi$ inherits the monotonicity from $\psi$. Moreover, since $\tau(\psi)<\alpha$, we have
\[\tau(\phi)=\tau(\psi)<\alpha.\]
In what follows, we bound the Hausdorff dimension of $K\cap W_d(\phi,\bta)$ from above, which immediately gives an upper bound for that of $K\cap W_d(\psi,\bta)$.

By the monotonicity of $\phi$, it is straightforward to verify that if $\bx\in K\cap W_d(\phi,\bta)$, then there exist infinitely many integers $m$ such that
\[\|q\bx-\bta\|<\phi(2^m)\quad\text{for some $2^m\le q<2^{m+1}$}.\]
Equivalently, $\bx\in K\cap A_{2^m}(\phi(2^m),\bta)$ for infinitely many $m$.
Therefore, we have
\[	K\cap W_d(\phi,\bta)\subset \bigcap_{M=1}^\infty\bigcup_{m=M}^\infty K\cap A_{2^m}\big(\phi(2^m),\bta\big).\]
Note that for each $M\ge 1$, the set $\bigcup_{m=M}^\infty K\cap A_{2^m}(\phi(2^m),\bta)$ is a cover of $K\cap W_d(\phi,\bta)$. Since $\phi(2^m)\ge 2^{-m\alpha}$ for all $m\ge 1$, Lemma \ref{l:global counting} is applicable to $K\cap A_{2^m}(\phi(2^m),\bta)$ whenever $m$ is sufficiently large. Therefore, for all sufficiently large $m$, it follows from Lemma \ref{l:global counting} that $K\cap A_{2^m}(\phi(2^m),\bta)$ can be covered by
\[
\asymp 2^{m(1+\delta)}\phi(2^m)^{d-\delta}
\]
balls of radius $2^{-m}\phi(2^m)$. Let $1/d<\tau<\tau(\phi)$. Then, for all sufficiently large $q$, we have $\phi(q)\le q^{-\tau}$.  For any $0<s\le\delta$, by the definition of $s$-dimensional Hausdorff measure,
\[\begin{split}
	\hm^s\big(K\cap W_d(\phi,\bta)\big)&\ll\liminf_{M\to\infty}\sum_{m=M}^{\infty} 2^{m(1+\delta)}\phi(2^m)^{d-\delta}\cdot\big(2^{-m}\phi(2^m) \big)^s\\
	&\le \liminf_{M\to\infty}\sum_{m=M}^{\infty} 2^{m(1+\delta)}2^{-m\tau(d-\delta)}\cdot(2^{-m}2^{-m\tau})^s\\
	&=\liminf_{M\to\infty}\sum_{m=M}^{\infty} 2^{m(1+\delta-\tau(d-\delta)-s(1+\tau))}.
\end{split}\]
Observe that the above $\liminf$ tends to zero if and only if the exponent satisfies
 \[1+\delta-\tau(d-\delta)-s(1+\tau)<0,\]
which is equivalent to
 \[s>\frac{-\tau(d-\delta)+1+\delta}{1+\tau}=\delta+\frac{1+d}{1+\tau}-d.\]
Therefore,
\[\hdim\big(K\cap W_d(\phi,\bta)\big)\le \delta+\frac{1+d}{1+\tau}-d.\]
Since $\tau<\tau(\phi)=\tau(\psi)$ is arbitrary, we conclude that
\[\hdim\big(K\cap W_d(\phi,\bta)\big)\le \delta+\frac{1+d}{1+\tau(\phi)}-d,\]
which completes the proof of \eqref{eq:upper bound} in Theorem \ref{t:main}.

 \section{Proof of Theorem \ref{t:main}: lower bound of $\hdim\big(K\cap W_d(\psi,\theta)\big)$}\label{s:lower bound}
The goal of this section is to prove the lower bound described in Theorem~\ref{t:main}. By the monotonicity of $\psi$, it is straightforward to verify that if $\bx\in K$ and there exist infinitely many integers $m$ such that
\[\|q\bx-\bta\|<\psi(2^{m+1})\quad\text{for some $2^m\le q<2^{m+1}$},\]
then $\bx\in K\cap W_d(\psi,\bta)$.
Therefore, we have
\begin{equation}\label{eq:lower inclusion}
\begin{split}
		K\cap W_d(\psi,\bta)&\supset \bigcap_{M=1}^\infty\bigcup_{m=M}^\infty K\cap A_{2^m}\big(\psi(2^{m+1}),\bta\big)\\
	&=\limsup_{m\to\infty} K\cap A_{2^m}\big(\psi(2^{m+1}),\bta\big).
\end{split}
\end{equation}
To establish a lower bound for the Hausdorff dimension (or measure) of
\(K \cap W_d(\psi,\bta)\), it suffices, by Theorem~\ref{t:weaken},
to show that for some suitable parameter \( s < \delta \),
\[\limsup_{m\to\infty}\hc^s\Big(B\cap A_{2^m}\big(\psi(2^{m+1}),\bta\big)\Big)\gg \mu(B)\]
holds for all ball $B$ in $K$, where the implied constant is independent of $B$.

Now, assume that $\mu$ satisfies the $(\alpha,\beta,\boldsymbol{\theta})$-local counting property for some $1/d<\alpha<\frac{1+\delta-\beta\delta}{d-\delta}$, $0<\beta<1$ and $\bta\in\R^d$.

Let $B\subset K$ be fixed. Choose $\tau$ so that $1/d<\tau(\psi)<\tau<\alpha$, this is possible due to our assumption $\tau(\psi)<\alpha$. Hence, for any $c>0$, there exists infinitely many $m$ such that
\begin{equation}\label{eq:assume}
	\psi(2^{m+1})>c2^{-m\tau}> 2^{-m\alpha}.
\end{equation}
Here we have introduced the constant $c>0$ so that, after completing the proof of \eqref{eq:lower bound} in Section \ref{ss:general case},
the same construction can be used with minimal modification to establish \eqref{eq:s} in Section \ref{ss:special case}.

Let $m \ge \log_2|B|^{-1/\beta}$ be an arbitrary  integer satisfying \eqref{eq:assume} and such that the local estimate \eqref{eq:local counting} can be applied with $Q$ replaced by $2^m$. Note that there are still infinitely many admissible choices of $m$. For simplicity, write
\begin{equation}
	Q=2^m\qaq F=B\cap A_Q(cQ^{-\tau},\bta).
\end{equation}
Define a probability measure supported on $F\subset B\cap A_{2^m}(\psi(2^{m+1}),\bta)$ by
\begin{equation}\label{eq:nu}
	\nu:=\frac{\mu|_{F}}{\mu(F)}.
\end{equation}
Since $m\ge \log_2 |B|^{-1/\beta}$ (equivalently $|B|\ge 2^{-m\beta}=Q^{-\beta}$) and $cQ^{-\tau}>Q^{-\alpha}$ (see \eqref{eq:assume}), by the $(\alpha,\beta,\bta)$-local counting property of $\mu$,
\begin{equation}\label{eq:mu(F)}
	\mu(F)=\mu\bigl(B \cap A_Q(cQ^{-\tau},\bta)\bigr)\stackrel{\eqref{eq:local counting}}{\asymp} \mu(B)\cdot Q^{1-\tau d}.
\end{equation}

Next, we estimate the $\nu$-measure of an arbitrary ball. Let $\bx \in F$ and let $0<r<|B|$. We proceed by considering several ranges of $r>0$.

\noindent {\bf Case 1}: $Q^{-\beta}\le r<|B|$. In this case, the $(\alpha,\beta,\bta)$-local counting property is applicable to the ball $B(\bx,r)$. By \eqref{eq:local counting} and the $\delta$-Ahlfors regularity of $\mu$, we have
\begin{equation}\label{eq:mu(B(x,r))}
	\mu\big(B(\bx,r)\cap A_Q(cQ^{-\tau},\bta)\big)\stackrel{\eqref{eq:local counting}}{\asymp} \mu\big(B(\bx,r)\big)\cdot Q^{1-\tau d}\asymp r^\delta\cdot Q^{1-\tau d}.
\end{equation}
Therefore,
\[\begin{split}
	\nu\big(B(\bx,r)\big)&\stackrel{\eqref{eq:nu}}{=}\frac{\mu|_{F}\big(B(\bx,r)\big)}{\mu(F)}\stackrel{\eqref{eq:mu(F)}}{\ll}\frac{\mu\bigl(B(\bx,r)\cap A_Q(cQ^{-\tau},\bta)\bigr)}{\mu(B)\cdot Q^{1-\tau d}}\stackrel{\eqref{eq:mu(B(x,r))}}{\asymp}\frac{r^\delta}{\mu(B)}\\
	& \le \frac{r^{\delta+\frac{1+d}{1+\tau}-d}}{\mu(B)},
\end{split}\]
since $\tau>1/d$.

\noindent {\bf Case 2}: $Q^{-\beta\delta/(\delta+\frac{1+d}{1+\tau}-d)}\le r< Q^{-\beta}$. Applying the estimate from Case 1 with $r=Q^{-\beta}$, we obtain
\[\nu\big(B(\bx,r)\big)\le\nu\big(B(\bx,Q^{-\beta})\big)\ll \frac{Q^{-\beta\delta}}{\mu(B)}\le \frac{r^{\delta+\frac{1+d}{1+\tau}-d}}{\mu(B)}.\]

The $(\alpha,\beta,\bta)$-local counting property allows us to control the $\nu$-measure of a ball $B(\bx,r)$ effectively when $r \ge Q^{-\beta\delta/(\delta+\frac{1+d}{1+\tau}-d)}$. However, this estimate breaks down at smaller scales. To deal with this difficulty, we will adopt different approaches depending on whether the conditions $\delta>d-1$,  $\bta=\bzero$ and $\tau<\frac{\delta+2-d}{2d-\delta-1}$ are satisfied. Accordingly, the analysis will be divided into separate cases in the following two subsections.

\subsection{Estimate of the $\nu$-measure of balls for general cases}\label{ss:general case}
\

\noindent {\bf Case 3A}: $0<r<Q^{-\beta\delta/(\delta+\frac{1+d}{1+\tau}-d)}$. At this scale, the estimate \eqref{eq:mu(B(x,r))} is not guaranteed to hold, and we instead have the following somewhat coarse estimate:
\[\begin{split}
	\nu\big(B(\bx,r)\big)&\stackrel{\eqref{eq:nu}}{=}\frac{\mu|_{F}\big(B(\bx,r)\big)}{\mu(F)}\stackrel{\eqref{eq:mu(F)}}{\ll}\frac{\mu\bigl(B(\bx,r)\bigr)}{\mu(B)\cdot Q^{1-\tau d}}\asymp\frac{r^\delta Q^{\tau d-1}}{\mu(B)}\\
	& \le \frac{r^{\delta-(\tau d-1)(\delta+\frac{1+d}{1+\tau}-d)/(\beta\delta)}}{\mu(B)}.
\end{split}\]
\begin{proof}[Completing the proof of \eqref{eq:lower bound} in Theorem \ref{t:main}]
	Let
	\[s_\tau=\min\biggl\{\delta-\frac{(\tau d-1)(\delta+\frac{1+d}{1+\tau}-d)}{\beta\delta},\ \delta+\frac{1+d}{1+\tau}-d\biggr\}.\]
	By Cases 1, 2 and 3A above, we have
	\[\nu\big(B(\bx,r)\big)\ll \frac{r^{s_\tau}}{\mu(B)}.\]
	Note that $\nu$ is supported on  $F=B\cap A_Q(cQ^{-\tau},\bta)\subset B\cap A_{2^m}(\psi(2^{m+1}),\bta)$. Therefore, by mass distribution principle (see Proposition \ref{p:MDP}),
	\[\hc^{s_\tau}\big(B\cap A_{2^m}(\psi(2^{m+1}),\bta)\big)\gg\mu(B).\]
	Since there are infinitely many admissible choices of \( m \)
	(see the choice of \( m \) at the beginning of this section),
	it follows that
	\[\limsup_{m\to\infty}\hc^{s_\tau}\big(B\cap A_{2^m}(\psi(2^{m+1}),\bta)\big)\gg \mu(B)\]
	holds for any ball $B\subset K$.
	Let \( \tau \downarrow \tau(\psi) \) and denote by \( s \) the limit of \( s_\tau \).
	Theorem~\ref{t:weaken} (1) together with \eqref{eq:lower inclusion} yields
\[K\cap W_d(\psi,\bta)\in \scgd^{s}(K),\]
and hence
\[\begin{split}
	&\hdim\big(K\cap W_d(\psi,\bta)\big)\\
	\ge& s=\min\biggl\{\delta-\frac{(\tau(\psi) d-1)(\delta+\frac{1+d}{1+\tau(\psi)}-d)}{\beta\delta},\ \delta+\frac{1+d}{1+\tau(\psi)}-d\biggr\}.\qedhere
\end{split}\]
\end{proof}
\subsection{Estimate of the $\nu$-measure of balls for $\delta \ge d-1$, $\bf\theta= 0$, and $\tau < \frac{\delta + 2 - d}{2d - \delta - 1}$}\label{ss:special case}
Parallel to Case 3A, we now continue the discussion from Case 2, dealing with the situation
\[
0 < r < Q^{-\beta\delta/(\delta + \frac{1+d}{1+\tau} - d)}
\]
by a method different from that in Section~\ref{ss:general case}. We divide the range of $r$ into three cases, which will be analyzed separately in Cases~3B--5B below. Prior to discussing these cases, we first state several results that will be used in the subsequent analysis.


\begin{lemma}[{\cite[Lemma 4]{KrisThornVelanibadAdv}}]\label{l:simplex lemma}
	Let $E \subset \R^d$ be a convex set with $d$-dimensional Lebesgue measure
	\[
	\lm_d(E) \le Q^{-(1+d)}/(d!).
	\]
	Then the rational points in $E$ with denominators $1 \le q \le Q$ all lie on some hyperplane in $\R^d$.
\end{lemma}

This lemma can be viewed as a higher-dimensional analogue of the classical one-dimensional fact: any interval in $\R$ of length $\le Q^{-2}$ contains at most one rational number $p/q$ with $1 \le q\le Q$.

In \cite[Proposition 6.3]{KleinbockWeissbadISJ}, Kleinbock and Weiss proved that any $\delta$-Ahlfors regular measure with $\delta > d-1$ is absolutely $(C,\delta+1-d)$-decaying for some constant $C$ depending only on $d$ and the measure $\mu$, in the following sense: for any ball $B(\bx,r)\subset\R^d$, any hyperplane $L\subset\R^d$, and any $\varepsilon>0$,
\begin{equation}\label{eq:absolutely decaying}
	\mu\bigl(B(\bx,r)\cap L^{(\varepsilon)}\bigr)
	\le
	C\biggl(\frac{\varepsilon}{r}\biggr)^{\delta+1-d}
	\mu\bigl(B(\bx,r)\bigr),
\end{equation}
where $L^{(\varepsilon)}$ denotes the $\varepsilon$-neighbourhood of $L$.

\noindent {\bf Case 3B}: $Q^{-(1+1/d)}/(8d!)\le r<Q^{-\beta\delta/(\delta+\frac{1+d}{1+\tau}-d)}$. Divide the ball $B(\bx,r)$ into smaller hypercubes of equal side length $Q^{-(1+1/d)}/(8d!)$. If any such hypercube $D$ has a nonempty intersection with $K$, then $2D \cap K$ contains a ball of radius $|D|$, where $2D$ denotes the hypercube with the same center as $D$ but twice the side length. Since any point in $K$ belongs to at most $3^d$ different sets of the form $2D \cap K$, a simple volume argument shows that the ball $B(\bx,r)$ can intersect at most
\[
\ll \frac{r^\delta}{Q^{-(1+1/d)\delta}}
\]
hypercubes $D$. Furthermore, since each $D$ is convex and its $d$-dimensional Lebesgue measure is bounded by $Q^{-(1+d)}/(8d!)^d\le (2Q)^{-(1+d)}/(d!)$, it follows from the simplex lemma above (see Lemma \ref{l:simplex lemma}) that the rational points with denominator $1\le q \le 2Q$ contained in $D$ lie in some hyperplane $L\subset\R^d$.
Therefore,
\begin{equation}\label{eq:hyperplane}
	D\cap A_Q(cQ^{-\tau},\bzero)\subset D\cap L^{(cQ^{-(1+\tau)})}.
\end{equation}
Using the absolutely decaying propertying of $\mu$ (see \eqref{eq:absolutely decaying}), we have
\[\begin{split}
	\mu|_{F}\big(B(\bx,r)\big)&\le \mu\big(B(\bx,r)\cap A_Q(cQ^{-\tau},\bzero)\big)\\
	&\ll \frac{r^\delta}{Q^{-(1+1/d)\delta}}\cdot \bigg(\frac{Q^{-(1+\tau)}}{Q^{-(1+1/d)}}\bigg)^{\delta+1-d} Q^{-(1+1/d)\delta}\\
	&=r^\delta Q^{(-\tau+1/d)(\delta+1-d)}.
\end{split}\]
It then follows from the definition of $\nu$  that
\[\begin{split}
	\nu\big(B(\bx,r)\big)&\stackrel{\eqref{eq:nu}}{=} \frac{\mu|_{F}\big(B(\bx,r)\big)}{\mu(F)}\stackrel{\eqref{eq:mu(F)}}{\ll}  \frac{\mu|_{F}\big(B(\bx,r)\big)}{\mu(B)\cdot Q^{1-\tau d}}\ll\frac{r^\delta Q^{(-\tau+1/d)(\delta+1-d)}}{\mu(B)\cdot Q^{1-\tau d}}\\
	&=\frac{r^\delta Q^{(\tau-1/d)(2d-1-\delta)}}{\mu(B)}\le \frac{r^{\delta-(\tau-1/d)(2d-1-\delta)(\delta+\frac{1+d}{1+\tau}-d)/(\beta\delta)}}{\mu(B)}.
\end{split}\]

\noindent {\bf Case 4B}: $cQ^{-(1+\tau)}\le r<Q^{-(1+1/d)}/(8d!)$. Since the $d$-dimensional Lebesgue measure of the ball $B(\bx,r)$ is less than $Q^{-(1+d)}/(8d!)^d$, by the same reason as \eqref{eq:hyperplane},
\[B(\bx,r)\cap A_Q(cQ^{-\tau},\bzero)\subset B(\bx,r)\cap L^{(cQ^{-(1+\tau)})}\]
for some hyperplane $L\subset\R^d$. Again, by the absolutely decaying property of $\mu$ and the definition of $\nu$, we have
\[\begin{split}
	\nu\big(B(\bx,r)\big)&\stackrel{\eqref{eq:nu}}{=} \frac{\mu|_{F}\big(B(\bx,r)\big)}{\mu(F)}\stackrel{\eqref{eq:mu(F)}}{\ll} \frac{\mu\bigl(B(\bx,r) \cap A_Q(cQ^{-\tau},\bzero)\bigr)}{\mu(B)\cdot Q^{1-\tau d}}\\
	&\stackrel{\eqref{eq:absolutely decaying}}{\ll} \frac{(Q^{-(1+\tau)}/r)^{\delta+1-d}\cdot r^\delta}{\mu(B)\cdot Q^{1-\tau d}}
	=\frac{r^{d-1} Q^{-(1+\tau)(\delta+1-d)-(1-\tau d)}}{\mu(B)}.
\end{split}\]
Since $\tau < \frac{\delta + 2 - d}{2d - \delta - 1}$, the exponent of $Q$ satisfies
\[\begin{split}
	-(1+\tau)(\delta+1-d)-(1-\tau d)=\tau(2d-\delta-1)-(\delta+2-d)<0.
\end{split}\]
By $cQ^{-(1+\tau)}\le r$, we have
\[Q^{-(1+\tau)(\delta+1-d)-(1-\tau d)}\ll r^{\delta+1-d+\frac{1-\tau d}{1+\tau}}.\]
It then follows that
\[\nu\big(B(\bx,r)\big)\ll \frac{r^{\delta+\frac{1+d}{1+\tau}-d}}{\mu(B)}.\]

\noindent {\bf Case 5B}: $0< r<cQ^{-(1+\tau)}$. In this case, we have
\[\begin{split}
	\nu\big(B(\bx,r)\big)&\stackrel{\eqref{eq:nu}}{=}\frac{\mu|_{F}\big(B(\bx,r)\big)}{\mu(F)}\stackrel{\eqref{eq:mu(F)}}{\ll} \frac{r^\delta}{\mu(B)\cdot Q^{1-\tau d}}\ll \frac{r^{\delta+\frac{1-\tau d}{1+\tau}}}{\mu(B)}=\frac{r^{\delta+\frac{1+d}{1+\tau}-d}}{\mu(B)}.
\end{split}\]
\begin{remark}\label{r:fail}
	To ensure that the exponent of $r$ in Case~3B satisfies
	\[
	\delta-\frac{(\tau-1/d)(2d-\delta-1)\bigl(\delta+\tfrac{1+d}{1+\tau}-d\bigr)}{\beta\delta}
	\ge
	\delta+\frac{1+d}{1+\tau}-d,
	\]
	it suffices to require that
	\[
	\tau\ge
	\frac{d\beta\delta-(2d-\delta-1)(1+\delta)}{(2d-\delta-1)(\delta-d)}.
	\]
	This condition is automatically satisfied if
	\[
	\frac{d\beta\delta-(2d-\delta-1)(1+\delta)}{(2d-\delta-1)(\delta-d)}
	\le \frac{1}{d}.
	\]
	A straightforward simplification of this inequality yields
	\[
	\delta \ge 2d-1-\frac{d^2\beta}{1+d}.
	\]
	Since $\delta<d$, this condition is non-vacuous provided that
	\[
	\beta > 1-d^{-2}.
	\]
When \( d = 1 \), there are indeed many measures satisfying this condition, since in this case \( 1-d^{-2}=0 \).
However, when \( d>1 \), even in the case of the Lebesgue measure, it remains unclear whether it satisfies this condition. Nevertheless, the above argument still provides some insight into how the higher-dimensional case might be approached.
\end{remark}
\begin{proof}[Completing the proof of \eqref{eq:lower measure} in Theorem \ref{t:main}]
	Let
\[s=\min\bigg\{\delta-\frac{(\tau-\frac{1}{d})(2d-\delta-1)(\delta+\frac{1+d}{1+\tau}-d)}{\beta\delta},\ \delta+\frac{1+d}{1+\tau}-d\bigg\}.\]
	By Cases 1, 2, 3B, 4B and 5B above, we have
	\[\nu\big(B(\bx,r)\big)\ll \frac{r^s}{\mu(B)}.\]
	Note that $\nu$ is supported on  $B\cap A_Q(cQ^{-\tau},\bzero)$. Therefore,
	\[\hc^{s}\big(B\cap A_Q(cQ^{-\tau},\bzero)\big)\gg\mu(B).\]
	Since $\psi(q)\ge cq^{-\tau}$ with $1/d<\tau<\min\{\alpha, \frac{\delta + 2 - d}{2d - \delta - 1}\}$ for infinitely many $q$, it follows that
	\[\limsup_{\substack{Q\to\infty\\ \psi(Q)\ge cQ^{-\tau}}}\hc^{s}\big(B\cap A_Q(cQ^{-\tau}, \bzero)\big)\gg \mu(B)\]
	holds for any ball $B\subset K$. By
	Theorem~\ref{t:weaken} (2),
	\[\limsup_{\substack{Q\to\infty\\ \psi(Q)\ge cQ^{-\tau}}} K\cap A_Q(cQ^{-\tau}, \bzero)\in \scgm^{s}(K).\]
	Since the set on the left is contained in $K\cap W_d(\psi,\bzero)$, we have
	\[K\cap W_d(\psi,\bzero)\in \scgm^s(K),\]
	and hence
	\[\hm^s\big(K\cap W_d(\psi,\bzero)\big)=\hm^s(K)=\infty,\]
	which completes the proof of \eqref{eq:lower measure} in Theorem \ref{t:main}.
\end{proof}
\section{Proof of Corollary \ref{c:corollary}}\label{s:corollary}
(1) By \eqref{eq:lower bound} in Theorem \ref{t:main}, for any $1/d<\tau<\alpha$, we have
\[\hdim\big(K\cap W_d(\psi_\tau,\bta)\big)\\
\ge\min\biggl\{\delta-\frac{(\tau d-1)(\delta+\frac{1+d}{1+\tau}-d)}{\beta\delta},\ \delta+\frac{1+d}{1+\tau}-d\biggr\}.
\]
As $\tau\downarrow 1/d$, the right-hand side converges to $\delta$. Consequently,
\[\lim_{\tau\downarrow1/d}\hdim\big(K\cap W_d(\psi_\tau,\bta)\big)\ge \delta.\]
Since the reverse inequality is immediate, it then follows that
\[\hdim\big(K\cap\VWA_d(\bta)\big)=\delta.\]

(2) Note that the assumptions $d=1$, $\theta=0$ and $\beta>0$ imply that \eqref{eq:lower measure} in Theorem \ref{t:main} can be applied to obtain the Hausdorff measure of $K\cap W_1(\tau,0)$. Observe that
\[E_1(\tau,0)\supset W_1(\psi_\tau,0)\Big\backslash\bigcup_{n=1}^\infty W_1(\psi_{\tau+1/n},0).\]
By the upper bound \eqref{eq:upper bound} in Theorem~\ref{t:main}, for each $n\in\N$,
\begin{equation}\label{eq:dim tau+1/n<}
	\hdim\big(K\cap W_1(\psi_{\tau+1/n},0)\big)\le \delta+\frac{2}{1+\tau+1/n}-1<\delta+\frac{2}{1+\tau}-1.
\end{equation}
On the other hand, by \eqref{eq:lower measure} in Theorem~\ref{t:main},
\[	\hm^s\big(K\cap W_1(\psi_\tau,0)\big)=\infty,\]
where
\[s=\min\bigg\{\delta-\frac{(\tau-1)(1-\delta)(\delta+\frac{2}{1+\tau}-1)}{\beta\delta},\ \delta+\frac{2}{1+\tau}-1\bigg\}.\]
A straightforward calculation shows that the minimum is
$\delta+\frac{2}{1+\tau}-1$ whenever
\[\tau>\frac{1-\delta^2-\beta\delta}{(1-\delta)^2}.\]
By our assumption $\delta\ge 1-\beta/2$, we have $\beta\ge 2-2\delta$, and hence
\[\frac{1-\delta^2-\beta\delta}{(1-\delta)^2}\le \frac{1-\delta^2-(2-2\delta)\delta}{(1-\delta)^2}=1<\tau,\]
which implies that $s=\delta+\frac{2}{1+\tau}-1$. It follows that
\[\begin{split}
	&\ \hm^{\delta+\frac{2}{1+\tau}-1}\big(K\cap E_1(\tau,0)\big)\\
	\ge&\, \hm^{\delta+\frac{2}{1+\tau}-1}\bigg(K\cap W_1(\psi_\tau,0)\Big\backslash\bigcup_{n=1}^\infty K\cap W_1(\psi_{\tau+1/n},0)\bigg)\\
	\ge&\, \hm^{\delta+\frac{2}{1+\tau}-1}\big(K\cap W_1(\psi_\tau,0)\big)-\sum_{n=1}^\infty \hm^{\delta+\frac{2}{1+\tau}-1}\big(K\cap W_1(\psi_{\tau+1/n},0)\big)\\
	\stackrel{\eqref{eq:dim tau+1/n<}}{=}&\hm^{\delta+\frac{2}{1+\tau}-1}\big(K\cap W_1(\psi_\tau,0)\big)=\infty.
\end{split}\]

\section{Some Examples}\label{s:example}
\subsection{Irreducible self-similar sets and proof of Theorem \ref{t:example1}}\label{ss:self similar}
Fix a finite set \(\Lambda\). An \emph{iterated function system} (IFS for short) is a finite collection
\[
\cf = \{\phi_i : i \in \Lambda\}
\]
of contractive similarities on \(\R^d\); that is, for each \(i \in \Lambda\), the map \(\phi_i\) has the form
\[
\phi_i(\bx) = \rho_i O_i \bx + \bb_i,
\]
where \(0 < \rho_i < 1\), \(O_i \in \SO_d(\R)\) is an orthogonal matrix, and \(\bb_i \in \R^d\). It is shown in \cite{Hutchinsoninitial} that there exists a unique non-empty compact set
\(K = K(\cf) \subset \R^d\), called the {\em self-similar set} associated with $\cf$, such that
\begin{equation}\label{eq:self similar set definition}
	K = \bigcup_{i \in \Lambda} \phi_i(K).
\end{equation}
Following \cite{Hutchinsoninitial}, we say that \(\cf\) satisfies the \emph{open set condition} (OSC for short) if there exists a non-empty bounded open set \(U \subset \R^d\) such that
\[
\bigcup_{i \in \Lambda} \phi_i(U) \subset U
\quad\text{and}\quad
\phi_i(U) \cap \phi_j(U) = \emptyset \quad \text{for all } i \ne j.
\]
If \(\cf\) is an IFS satisfying the OSC, then the Hausdorff dimension of \(K\) is given by \(\delta\), which is the unique solution to
\[
\sum_{i \in \Lambda} \rho_i^\delta = 1,
\]
and moreover,
\[
0 < \hm^\delta(K) < \infty.
\]
Let \(\mu\) denote the normalized restriction of \(\hm^\delta\) to \(K\), which we refer to as the {\em self-similar measure}. Then \(\mu\) is \(\delta\)-Ahlfors regular and invariant under \(\cf\) in the sense that
\begin{equation}\label{eq:self-similar measure}
	\mu = \sum_{i \in \Lambda} \rho_i^\delta \cdot (\mu \circ \phi_i^{-1}).
\end{equation}
Let \(\Lambda^* = \bigcup_{k \ge 1} \Lambda^k\) denote the set of finite words over \(\Lambda\). Let $k\ge 1$. Given a word \(\omega = (i_1,\dots,i_k) \in \Lambda^k\), we define
\begin{equation}\label{eq:branch}
	\begin{split}
		&\phi_\omega := \phi_{i_1} \circ \cdots \circ \phi_{i_k}, \quad \rho_\omega:=\rho_{i_1}\cdots\rho_{i_k},\\
		&K_\omega := \phi_\omega(K) \quad \text{and} \quad
		\mu_\omega := \frac{\mu|_{K_\omega}}{\mu(K_\omega)} = \mu \circ \phi_\omega^{-1}.
	\end{split}
\end{equation}
We call \(\mu_\omega\) a \emph{branch} of \(\mu\). Intuitively, \(\mu_\omega\) represents the measure \(\mu\) restricted to a ``cylinder" set \(K_\omega\) and then rescaled to a probability measure.

Finally, an IFS \(\cf\) is called \emph{strongly irreducible} if \(\R^d\) is the only finite union of affine subspaces of \(\R^d\) that is invariant under all \(\phi_i \in \cf\). This condition ensures that the system cannot be confined to a lower-dimensional structure and is crucial for certain rigidity and dimension results.

The proof of the local estimate (see Corollary \ref{c:self-similar set} below) for rational numbers near strong irreducible self-similar set relies on the equidistribution of self-similar measure in the homogeneous space $X=\SL_{d+1}(\mathbb{R})/\SL_{d+1}(\mathbb{Z})$, which has been proved in \cite{BenHeZhangfractalJAMS,BenHeZhangfractalRd,KhalilLuethifractalInvent}.

Set $\mathfrak g := \mathfrak{sl}_{d+1}(\mathbb R)$ to be the Lie algebra of $G=\SL_{d+1}(\R)$.
We equip $G$ with the unique right $G$-invariant Riemannian metric which
coincides with $\|\cdot\|_{\mathfrak g}$ at $\mathfrak g = T_{\mathrm{Id}}G$.
Write
\[
\mathcal A
=
\{E_{i,j} : 1\le i,j\le d+1,\ i\ne j\}
\cup
\{E_{i,i}-E_{i+1,i+1} : i=1,\dots,d\}
\]
for the standard basis of $\mathfrak g$.
Given $l\in\mathbb N$, write $\Xi_l$ for the set of words of length $l$ with
letters in $\mathcal A$. Each $D\in\Xi_l$ acts as a differential operator on
the space of smooth functions $C^\infty(X)$. Let $m_X$ be the $\SL_{d+1}(\R)$-invariant Haar measure on $X$.
Given $f\in C^\infty(X)$ and $p\in[1,\infty]$, we set
\[
\cs_{p,l}(f)
=
\sum_{D\in\Xi_l} \|Df\|_{L^p},
\]
where $\|\cdot\|_{L^p}$ refers to the $L^p$-norm for $m_X$ on $X$.
Let
\[B^\infty_{p,l}(X)=\{f\in C^\infty(X):\mathcal{S}_{p,l}(f)<\infty\}.\]

 For $t\in\R$ and $\bx\in\R^d$, let
\begin{equation}\label{eq:g and u}
	g(t)=
	\begin{pmatrix}
		e^{t/d}\I_d&0\\
		0&e^{-t}
	\end{pmatrix}
	\qaq
	u(\bx)=
	\begin{pmatrix}
		\I_d&\bx\\
		0&1
	\end{pmatrix},
\end{equation}
where $\I_d$ is the $d\times d$ identity matrix.

Let $\exp(\cdot)$ denote the exponential map from $\mathfrak{g}$ to $G$.  Denote by $B_{r}^\mathfrak{g}(0)$ the ball in $\mathfrak{g}$ centered at 0 with radius $r$, and let $r_0$ be such that $\exp: B^\mathfrak{g}_{r_0}(0)\rightarrow \mathrm{exp}(B_{r_0}^\fg(0))$ is a diffeomorphism.  For $y\in X$, we define the injectivity radius at $y$ as
\[ \inj(y)=\sup\{ r\le r_0: \text{ the map } g\mapsto gy \text{ is injective on }\exp(B_r^\mathfrak{g}(0) )\}.\]
 This definition of injectivity radius is different from that in \cite{BenHeZhangfractalJAMS,BenHeZhangfractalRd}, but this will not affect the final result.
Next, we recall the following key equidistribution result from~\cite{BenHeZhangfractalRd}, which generalizes earlier special cases considered in~\cite{BenHeZhangfractalJAMS, KhalilLuethifractalInvent}. We remark  that their result applies to a more general class of self-similar measures than the one defined in~\eqref{eq:self-similar measure}. However, such a level of generality is not needed for our purposes, and we therefore do not present it in full generality here.
\begin{theorem}[{\cite[Theorem 1.2]{BenHeZhangfractalRd}}]
\label{t:effective equidistribution}
Let $K$ be a strongly irreducible self-simlar set of an IFS $\cf$ satisfying the OSC and let $\mu$ be the corresponding $\delta$-Ahlfors regular self-similar measure. There exists a constant $\kappa>0$ such that for all $t>0$, $y\in X$, and $f\in B^\infty_{p,l}(X)$, the following holds
\[\bigg|\int_{\R^d}f(g(t)u(\bx)y)\dif\mu(\bx)-\int_{X}f \dif m_X\bigg|\ll \inj(y)^{-1}\cs_{p,l}(f)e^{-\kappa t},\]
where the implicit constant only depends on the IFS $\cf$ and $l=\lceil \frac{d(d+1)}{4}\rceil$.
Moreover, the exponent $\kappa$ can be chosen so that it is nondecreasing as $\hdim K$ increases.
\end{theorem}
\begin{remark}
	The ``Moreover" part is not explicitly stated in their paper, but it can be verified to hold. In fact, $\kappa$ is very small, while the proof of the zero--one law (See Theorem BHZ) only relies on its existence. This is one reason why the authors of \cite{BenHeZhangfractalJAMS,BenHeZhangfractalRd} did not explicitly derive $\kappa$.
\end{remark}

As a corollary of the above theorem, we obtain an effective equidistribution result for the branches of $\mu$. The proof is based on the approach of Chen~\cite{Chenfractal} in the one-dimensional setting and is included for completeness.

\begin{corollary}[{\cite[Corollary 2.6]{Chenfractal}}]\label{c:local effective equidistribution}
Let $K$ be a strongly irreducible self-simlar set of an IFS $\cf$ satisfying the OSC and let $\mu$ be the corresponding $\delta$-Ahlfors regular self-similar measure. Let $f\in B_{p,l}^\infty(X)$. Then, for any $\omega\in\Lambda^*$ and $y\in X$, we have
\begin{align*}
\bigg|\int_{\R^d} f(g(t)u(\bx)y)\dif\mu_\omega(\bx)-\int_X f \dif m_X\bigg|\ll \inj(y)^{-1}\rho_\omega^{-1-\frac{d\kappa}{d+1}}\cs_{p,l}(f)e^{-\kappa t},
\end{align*}
where $\kappa$ is as in Theorem~\ref{t:effective equidistribution} and the implicit constant only depends on the IFS $\mathcal{F}$.
\end{corollary}
\begin{proof}
Since $\mu_\omega=\mu\circ\phi_\omega^{-1}$ and $\phi_\omega(\bx)=\rho_\omega\bx+\bb_\omega$ with $\bb_\omega=\phi_\omega(\bf0)$, we have
\begin{align*}
&\int_{\R^d} f(g(t)u(\bx) y)\dif \mu_\omega(\bx)\\
\stackrel{\eqref{eq:branch}}{=}&\int_{\R^d} f(g(t)u(\bx) y)\dif \mu\circ \phi_\omega^{-1}(\bx)=\int_{\R^d} f(g(t) u(\rho_\omega \bx+\bb_\omega)y)\dif \mu(\bx)\\
=&\int_{\R^d} f\left(g\left(t+ \frac{d}{d+1}\ln \rho_\omega\right)u(\bx)g\left(-\frac{d}{d+1}\ln \rho_\omega\right)u(\bb_\omega)y\right)\dif\mu(\bx),
\end{align*}
where the last equality follows from
\[u(\rho_\omega \bx+\bb_\omega)=u(\rho_\omega \bx)u(\bb_\omega)\stackrel{\eqref{eq:g and u}}{=}g\left(\frac{d}{d+1}\ln \rho_\omega\right)u(\bx)g\left(-\frac{d}{d+1}\ln \rho_\omega\right)u(\bb_\omega).\]
Therefore, by Theorem \ref{t:effective equidistribution},
\begin{align*}
&\bigg|\int_{\R^d} f(g(t)u(\bx) y)\dif \mu_\omega(\bx)-\int_X f \dif m_X\bigg|\\
\ll\ &\inj\left( g\left(-\frac{d}{d+1}\ln \rho_\omega\right)u(\bb_\omega)y\right)^{-1}\cs_{p,l}(f)e^{-\kappa(t+\frac{d}{d+1}\ln \rho_\omega)}.
\end{align*}
By a standard estimate, there exists a constant \(C_1>0\) such that for any \(h\in G\) and \(z\in X\),
\[
\inj(hz) \ge C_1 \, \|\mathrm{Ad}(h)\|_{\infty}^{-1} \, \inj(z),
\]
where \(\mathrm{Ad}\) denotes the adjoint action of $G$ on \(\fg\) and $\|\cdot\|_{\infty}$ is the operator norm. Applying this with
\(h = g(-\frac{d}{d+1}\ln \rho_\omega) u(\bb_\omega)\) yields
\[\begin{split}
	\inj\left(g\left(-\frac{d}{d+1}\ln \rho_\omega\right)u(\bb_\omega)y\right)&\ge C_1\left\|\Ad\left(g\left(-\frac{d}{d+1}\ln \rho_\omega\right)\right)\right\|_\infty^{-1}\inj(u(\bb_\omega)y)\\
	&\ge C_1\,\rho_\omega\cdot\inj(u(\bb_\omega)y).
\end{split}\]
Since \(u(\bb_\omega)\) is uniformly bounded over \(\omega\in \Lambda^*\), there exists a constant \(C_2>0\) such that
\(\inj(u(\bb_\omega)y) \ge C_2 \, \inj(y)\). Combining these bounds, we obtain
\[\inj\left( g\left(-\frac{d}{d+1}\ln \rho_\omega\right)u(\bb_\omega)y\right)^{-1}\le C_1^{-1}C_2^{-1} \rho_\omega^{-1}\inj(y)^{-1},\]
which immediately implies the desired effective equidistribution result.
\end{proof}
By approximating $A_Q(\eta,\bzero)$ with smooth functions and applying the above effective equidistribution result, Khalil and Luethi \cite{KhalilLuethifractalInvent} obtained the following estimate, which is stated in a slightly different form from their original result for our purposes.

\begin{lemma}[{\cite[Theorem 9.1 and Lemma 12.7]{KhalilLuethifractalInvent}}]\label{l:global estimate}
\label{A_n}
Let $K$ be a strongly irreducible self-simlar set of an IFS $\cf$ satisfying the OSC and let $\mu$ be the corresponding $\delta$-Ahlfors regular self-similar measure. Let $Q\ge 1$ be a large integer and $0<\eta\le Q^{-1/d}$. Then, there exists a constant $C_\cf$ depending only on $\cf$ such that
\[Q\eta^d-C_\cf(Q^{-1}\eta)^{\frac{d\kappa}{d+1}}\ll\mu\left(A_Q(\eta,\bf 0)\right)\ll  Q\eta^d+C_\cf(Q^{-1}\eta)^{\frac{d\kappa}{d+1}},\]
where $\kappa$ is as in Theorem~\ref{t:effective equidistribution} and the implied constants are independent of $\cf$.
\end{lemma}
\begin{remark}\label{r:KL}
	In fact, Khalil and Luethi \cite{KhalilLuethifractalInvent} proved a result stronger than that stated in Lemma~\ref{l:global estimate}. Let
	\[\begin{split}
		E_Q(\eta,\bzero):=\bigg\{\bx\in\R^d:\bigg|\bx-\frac{\bp}{q}\bigg|<\frac{\eta}{q}&\text{ for some $\frac{\bp}{q}\in\Q^d$} \\
			&\text{with $0< q<2Q$ and $\gcd(q,\bp)=1$}\bigg\}.
	\end{split}\]
	Clearly,
	\[E_Q(\eta,\bzero)\setminus E_{Q/2}(\eta,\bzero)\subset A_Q(\eta,\bzero)\subset E_Q(\eta,\bzero).\]
	The upper and lower bounds given in Lemma~\ref{l:global estimate} are obtained precisely by estimating the $\mu$-measures of these two outer sets, as respectively given in \cite[Theorem 9.1 and Lemma 12.7]{KhalilLuethifractalInvent}.
\end{remark}

Motivated by this result together with Corollary~\ref{c:local effective equidistribution}, Chen~\cite{Chenfractal} derived the following estimate for $\mu_\omega(A_Q(\eta,\bf0))$. Although Chen's formulation is somewhat different from ours, it essentially implies the same result.

\begin{proposition}[{\cite[Proposition 2.8]{Chenfractal}}]
\label{p:local estimate}
Let $K$ be a strongly irreducible self-simlar set of an IFS $\cf$ satisfying the OSC and let $\mu$ be the corresponding $\delta$-Ahlfors regular self-similar measure. Let $Q\ge 1$ be a large integer and $0<\eta\le Q^{-1/d}$. Then, there exists a constant $C_\cf$ depending only on $\cf$ such that for any $\omega\in\Lambda^*$,
\[Q\eta^d-C_\cf\,\rho_\omega^{-1-\frac{d\kappa}{d+1}}(Q^{-1}\eta)^{\frac{d\kappa}{d+1}}\ll\mu_\omega\big(A_Q(\eta,{\bf 0})\big)\ll  Q\eta^d+C_\cf\,\rho_\omega^{-1-\frac{d\kappa}{d+1}}(Q^{-1}\eta)^{\frac{d\kappa}{d+1}},\]
where $\kappa$ is as in Theorem~\ref{t:effective equidistribution} and the implied constants are independent of $\cf$.
\end{proposition}
The main term $Q\eta^d$ in Proposition \ref{p:local estimate} dominates provided that
\[Q\eta^d\ge 2 C_\cf\, \rho_\omega^{-1-\frac{d\kappa}{d+1}}(Q^{-1}\eta)^{\frac{d\kappa}{d+1}}\quad\Longleftrightarrow\quad\rho_\omega\ge 2C_\cf\, Q^{-1}\eta^{-(d-\frac{d\kappa}{d+1})/(1+\frac{d\kappa}{d+1})}.\]
In particular, if $\eta\ge Q^{-\alpha}$ with $\alpha>1/d$, then the above inequalities hold if
\begin{equation}\label{eq:rhoboa}
	\rho_\omega\ge 2C_\cf\,Q^{-1} Q^{\alpha(d-\frac{d\kappa}{d+1})/(1+\frac{d\kappa}{d+1})}=2C_\cf\,Q^{-(1+\frac{d\alpha(d+1)\kappa}{d+1+d\kappa}-d\alpha)}.
\end{equation}
	Therefore, Proposition~\ref{p:local estimate} yields the desired local estimate for rational points in a neighbourhood of self-similar sets.

\begin{corollary}\label{c:self-similar set}
	Let $K$ be a strongly irreducible self-simlar set of an IFS $\cf$ satisfying the OSC and let $\mu$ be the corresponding $\delta$-Ahlfors regular self-similar measure. Choose $\alpha$ and $\beta$ so that
	\[\alpha>1/d\qaq 0<\beta<1+\frac{d\alpha(d+1)\kappa}{d+1+d\kappa}-d\alpha,\]
	where $\kappa$ is as in Theorem~\ref{t:effective equidistribution}.
	Then, $\mu$ satisfies the $(\alpha,\beta,\bzero)$-local counting property.
\end{corollary}
\begin{proof}
	Since $K$ satisfies the open set condition, any ball $B$ contains a branch
	$K_{\omega}$ of $K$ with $|B|\asymp |K_{\omega}|$, and is itself contained in a finite
	union of branches of $K$, namely
	\[
	B \subset \bigcup_{i=1}^{i_B} K_{\omega_i},
	\]
	where $i_B \asymp 1$ and $|B|\asymp |K_{\omega_i}|$ for all $i$. By the self-simlarity and the Ahlfors regularity of $\mu$, we have
	\[|B|\asymp|K_\omega|\quad\Longrightarrow\quad\mu(B)\asymp\mu(K_\omega).\]
	Therefore, the corollary follows if,  for any $\omega\in \Lambda^*$ with $\rho_\omega\ge Q^{-\beta}$ and $Q^{-\alpha}\le\eta\le Q^{-1/d}$,
	\[\mu\big(K_\omega\cap A_Q(\eta,{\bf0})\big)\asymp \mu(K_\omega)\cdot Q\eta^d\]
	whenever $Q$ is sufficiently large (depending only on $\alpha$ and $\beta$).
	This is indeed the case, since by Proposition~\ref{p:local estimate} and \eqref{eq:rhoboa},
	\[Q\eta^d\asymp \mu_\omega\big(A_Q(\eta,{\bf 0})\big)\stackrel{\eqref{eq:branch}}{=}\frac{\mu\big(K_\omega\cap A_Q(\eta,\bzero)\big)}{\mu(K_\omega)}.\qedhere\]
\end{proof}

\begin{proof}[Proof of Theorem \ref{t:example1}]
	(1) By Corollary \ref{c:self-similar set}, the $\delta$-Ahlfors regular self-similar measure $\mu$ supported on $K$ satisfies the $(\alpha,\beta,\bzero)$-local counting property, where
			\[\alpha>1/d\qaq 0<\beta<1+\frac{d\alpha(d+1)\kappa}{d+1+d\kappa}-d\alpha,\]
	and $\kappa$ is as in Theorem~\ref{t:effective equidistribution}. The conclusion is an immediate consequence of Corollaries~\ref{c:corollary} (1).

	(2) Suppose that $d=1$ and $\theta=0$. Denote by $\kappa(K)$ the exponent in Theorem \ref{t:effective equidistribution} associated to the self-similar set $K$. For any $\kappa_0>0$, choose $\alpha=\alpha(\kappa_0)>1$ such that
	\[\beta:=1+\frac{2\alpha\kappa_0}{2+\kappa_0}-\alpha>0.\]
	This is always possible, since as $\alpha\downarrow 1$ the left-hand side converges to $\frac{2\kappa_0}{2+\kappa_0}>0$.
	Suppose that $\kappa_0>0$ is small enough so that there are nontrivial self-similar sets $K$ satisfying $\kappa(K)> \kappa_0$. Let $0<\delta_0<1$ be such that $\delta_0>1-\beta/2$ and $\kappa(K)>\kappa_0$ whenever $\hdim K\ge \delta_0$. The second restriction is possible because, by Theorem~\ref{t:effective equidistribution},
	the parameter $\kappa(K)$ is non-decreasing as the Hausdorff dimension of $K$
	increases.

	Let $K\subset\R $ be a self-similar set with $\hdim K=\delta\ge \delta_0$. By the definition of $\delta_0$, we have $\kappa(K)>\kappa_0$, and so
	\[1+\frac{2\alpha\kappa(K)}{2+\kappa(K)}-\alpha> 1+\frac{2\alpha\kappa_0}{2+\kappa_0}-\alpha=\beta.\]
	By Corollary \ref{c:self-similar set}, $\mu$ satisfies the $(\alpha,\beta,0)$-local counting property with $\alpha$ and $\beta$ as defined above.
	Since $\delta\ge \delta_0\ge  1-\beta/2$, it follows from Corollary \ref{c:corollary} (2) that for any $1<\tau<\alpha$,
	\[\hm^{\delta+\frac{2}{1+\tau}-1}\big(K\cap E_1(\tau,0)\big)=\infty.\qedhere\]
\end{proof}

\subsection{Missing digits sets and proof of Theorem \ref{t:example2}}\label{ss:missing digit} In \cite{Yufractals}, Yu employed Fourier-analysis methods to show that sufficiently thick missing digits measures supported on $\R $ satisfy the $(\alpha,\beta,0)$-local counting property.
We begin by recalling some relevant definitions.

Let $b > 2$ be an integer, and let
\[
\cd \subset \{0,1,\dots,b-1\}^d
\]
be a non-empty proper subset.
The associated \emph{missing digits set} is defined by
\[
K(b,\cd)
:= \left\{ \bx \in [0,1]^d:
\bx = \sum_{n=1}^{\infty} \frac{\bd_n}{b^n},
\ \bd_n \in \cd \text{ for all } n \in \N
\right\}.
\]
Let $\mu_{b,\cd}$ be the natural normalized Hausdorff measure supported on $K(b,\cd)$, which will be referred to as the {\em missing digits measure}.

For a probability measure $\mu$ supported on $\R^d$, its {\it Fourier $l^1$-dimension} is defined by
\[\fldim\mu:=\sup\left\{s\ge 0:\sum_{\bxi\in\Z^d:|\bxi|\le M}|\widehat\mu(\bxi)|\ll M^{d-s}\text{ for all $M>0$}\right\},\]
where
\[\widehat\mu(\bxi):=\int_{\R^d}e^{-2\pi i\bxi \cdot \bx}\dif\mu(\bx)\]
is the Fourier transformation of $\mu$ and $\bxi\cdot\bx$ denotes the inner product.

Let
\begin{equation}\label{eq:A(q,eta,theta)}
	A(q,\eta,\bta):=\{\bx\in\R^d:\|q\bx-\bta\|<\eta\}.
\end{equation}
In \cite[\S 6.2]{Yufractals}, Yu proved that for some $\ve>0$ (which can be chosen arbitrarily small but, once fixed, remains constant), as long as $Q$ is sufficiently large and
\[1>\eta\ge Q^{1-1/(d-\fldim\mu)+\ve},\]
then
\begin{equation}\label{eq:GCP}
	\sum_{q=Q}^{2Q}\mu\big(A(q,\eta,\bta)\big)\asymp Q\eta^d.
\end{equation}
Since we will require that $\eta<Q^{-1/d}$, the above estimate is meaningful only when
\[1-\frac{1}{d-\fldim\mu}<-1/d\quad\Longrightarrow \quad \fldim\mu>d-\frac{d}{d+1}.\]
The above analysis indicates that for measures $\mu$ with sufficiently large
$\fldim\mu$, one has a good counting estimate for rational points
in neighborhoods of the corresponding fractal set. However, determining
$\fldim\mu$ is in general a highly nontrivial problem. In \cite{Yufractals,Yumanifold}, Yu calculated $\fldim\mu_{b,\cd}$ for the missing digits measure and in particular showed that for any $t\ge 1$,
\[\liminf_{\substack{b\to\infty\\ \#\cd\ge b^d-t}}\fldim \mu_{b,\cd}=d.\]
This establishes the existence of nontrivial fractal measures with arbitrarily large Fourier $l^1$-dimension.

From now on, we restrict our attention to the missing digits sets in $\R$, and for simplicity we denote $K(b,\cd)$ and $\mu_{b,\cd}$ by $K$ and $\mu$, respectively. The restriction to one dimension is simply because, at present, we are unable to establish the local counting property for missing digits measures in $\R^d$ with $d \ge 2$.

For any $\omega\in\cd^*:=\bigcup_{k\ge 1}\cd^k$, denote by $K_\omega$ and $\mu_\omega$ the branch of $K$ and $\mu$ associated to $\omega$, repectively (see \eqref{eq:branch} for precise definitions).
 To prove a local version of \eqref{eq:GCP}, observe that for any $\omega\in\cd^*$,
\begin{equation}\label{eq:branch Fourier}
	|\widehat{\mu_\omega}(\xi)|=|\widehat\mu(b^{-|\omega|}\xi)|,
\end{equation}
where $|\omega|$ denotes the length of $\omega$.
Consequently, it is shown in \cite[Pages 40-41]{Yufractals} that for any $1/2 < \gamma < \fldim\mu$, and any $\alpha$ and $\beta$ satisfying
$1 < \alpha < \frac{\gamma}{1-\gamma}$ and $0 < \beta < \frac{1-(1+\alpha)(1-\gamma)}{\gamma}$, we have
\begin{equation}\label{eq:local estimate inhomogeneous1}
	\sum_{q=Q}^{2Q}\mu_\omega\big(A(q,\eta,\theta)\big)\asymp Q\eta
\end{equation}
whenever $Q$ is large enough (depending on $\alpha$ and $\beta$ only), $\theta\in\R$, $Q^{-\alpha}\le\eta\le Q^{-1}$ and $|K_\omega|\ge Q^{-\beta}$. Note that $\mu_\omega=\frac{\mu|_{K_\omega}}{\mu(K_\omega)}$ for all $\omega\in\cd^*$ (see \eqref{eq:branch}). It follows that
\begin{equation}\label{eq:local estimate inhomogeneous2}
	\sum_{q=Q}^{2Q}\mu\big(K_\omega\cap A(q,\eta,\theta)\big)=\mu(K_\omega)\cdot \sum_{q=Q}^{2Q}\mu_\omega\big(A(q,\eta,\theta)\big)\asymp \mu(K_\omega)\cdot Q\eta.
\end{equation}
Although $A_Q(\eta,\theta)=\bigcup_{q=Q}^{2Q}A(q,\eta,\theta)$, this does not imply that
\[\mu\big(K_\omega\cap A_Q(\eta,\theta)\big)\asymp \mu(K_\omega)\cdot Q\eta\]
because the sets $A(q,\eta,\theta)$ may have overlaps. In \cite[Pages 57-58]{Yufractals}, after a carefully analysis of the measure of the overlaps, Yu was able to establish the above result for $\theta=0$. For arbitrary $\theta\in\R$, although some of his ideas are applicable, there are still several nontrivial difficulties.

%
%
\begin{lemma}\label{l:local inhomo}
	Suppose that $\fldim\mu>1/2$. For any $1/2<\gamma<\fldim\mu$, let $1<\alpha<\frac{\gamma}{1-\gamma}$ and $0<\beta<\frac{1-(1 +\alpha)(1-\gamma)}{\gamma}$. Then, for any $\theta\in \R$, $\mu$ satisfies the $(\alpha,\beta,\theta)$-local counting property.
\end{lemma}
\begin{proof}
	Let $\theta \in \R$. To prove the lemma, it is enough to show that
	\begin{equation}\label{eq:GCP local}
		\mu\big(K_\omega \cap A_Q(\eta,\theta)\big) \asymp \mu(K_\omega) \cdot Q \eta,
	\end{equation}
	whenever $Q$ is sufficiently large (depending only on $\alpha$ and $\beta$),  $Q^{-\alpha} \le \eta \le Q^{-1}$, and $|K_\omega| \ge Q^{-\beta}$.

	The upper bound of \eqref{eq:GCP local} follows directly from \eqref{eq:local estimate inhomogeneous2}, so it remains to establish the corresponding lower bound. Define the function
	\[f_Q=\sum_{q=Q}^{2Q}\cha_{A(q,\eta,\theta)}.\]
	Note that $f_Q$ is a non-negative integer-valued function. If $f_Q(x)=k>0$, then
	$x\in A(q,\eta,\theta)$ for precisely $k$ integers $q$ with $Q\le q\le 2Q$. In particular, $x\in A_Q(\eta,\theta)$. Therefore,
	\begin{align}
		0<f_Q(x)<k\quad&\Longrightarrow\quad f_Q(x)/k<1=\cha_{A_Q(\eta,\theta)} (x)\notag\\
		&\Longrightarrow\quad \frac{\int_{\{f_Q<k\}}f_Q\dif\mu_\omega}{k} \le \mu_\omega\big(A_Q(\eta,\theta)\big) \label{eq:f_Q/k}.
	\end{align}

	We claim that there exists $\rho\in(0,1)$ such that for any $2\le k<Q$,
	\begin{equation}\label{eq:claim}
		\int_{\{f_Q\ge k\}}f_Q\dif\mu_\omega\ll \frac{Q \eta}{k} + Q^\rho \eta \log_2 Q,
	\end{equation}
	where the implied constant is independent of $Q, k,\theta$ and $\eta$. This gives that for $k$ and $Q$ large enough,
	\[\begin{split}
		\mu\big(K_\omega\cap A_Q(\eta,\theta)\big)&\stackrel{\eqref{eq:branch}}{=}\mu(K_\omega)\cdot \mu_\omega\big(A_Q(\eta,\theta)\big)\stackrel{\eqref{eq:f_Q/k}}{\ge} \mu(K_\omega)\cdot  \frac{\int_{\{f_Q<k\}}f_Q\dif\mu_\omega}{k}\\
		&\gg\mu(K_\omega)\cdot \frac{Q\eta}{k},
	\end{split}\]
	where the last inequality follows from
	\[\begin{split}
		\int_{\{f_Q<k\}}f_Q\dif\mu_\omega&=\int f_Q\dif\mu_\omega-\int_{\{f_Q\ge k\}}f_Q\dif\mu_\omega\\
		&=\sum_{q=Q}^{2Q}\mu_\omega\big(A(q,\eta,\theta)\big)-\int_{\{f_Q\ge k\}}f_Q\dif\mu_\omega\\
		&\stackrel{\eqref{eq:local estimate inhomogeneous1},\eqref{eq:claim}}{\gg}Q\eta-c\bigg(\frac{Q \eta}{k} + Q^\rho \eta \log_2 Q\bigg)\asymp Q\eta
	\end{split}\]
	whenever $k$ and $Q$ are large enough.

	Now, we prove the claim, and in the course of the proof it will become clear how the parameters $\alpha,\beta,\gamma$ in the lemma arise. Let $Q^{-\alpha}\le \eta < Q^{-1}/4$. Once the lower bound in \eqref{eq:GCP local} is established for this case, the corresponding lower bound for $\eta$ in the range $Q^{-1}/4 \le \eta \le Q^{-1}$ follows immediately. Let $x\in \R$. If $f_Q(x)=m\ge 2$, then there exist exactly distinct $Q\le q_1<\cdots<q_m\le 2Q$ such that
	\[\|q_ix-\theta\|<\eta\quad\text{for $1\le i\le m$}.\]
	Consequently, for $1\le i\le m-1$
	\[\|(q_m-q_i)x\|=\|q_mx-\theta-(q_ix-\theta)\|\le 2\eta.\]
	This implies that $\|qx\|<2\eta$ has $m-1$ solutions for $q\le Q$. Let $q_{\max}$ be the largest integer for which $\|q_{\max} x\|<2\eta$, and let $a/n$ be the reduced fraction of $p_{\max}/q_{\max}$.  Since the distance between any two distinct rationals $p/q$ and $p'/q'$ is at least $1/(qq')$, we see that if $\|qx\|<2\eta, \|q'x\|<2\eta$ with $q>q'$, then
	\begin{equation}\label{eq:fraction equal}
		\bigg|x-\frac{p}{q}\bigg|<\frac{2\eta}{q}\text{ and }\bigg|x-\frac{p'}{q'}\bigg|<\frac{2\eta}{q'}\quad\Longrightarrow\quad \bigg|\frac{p}{q}-\frac{p'}{q'}\bigg|<\frac{4\eta}{q'}.
	\end{equation}
	Since $\eta<Q^{-1}/4$, the last inequality is less than $1/(qq')$, and so $p/q=p'/q'$. Therefore, any rational number $p/q$ satisfying $|x-p/q|<2\eta/q$ has the form $\frac{at}{nt}$ for some integer $t\ge 1$. Since there are $m-1$ many such rationals $p/q$ and all the denumerator is less than $Q$, we have that $1\le n\le q_{\max}/(m-1)\le Q/(m-1)$. Hence,
	\begin{equation}\label{eq:reduce}
		\bigg|x-\frac{a}{n}\bigg|=\bigg|x-\frac{p_{\max}}{q_{\max}}\bigg|<\frac{2\eta}{q_{\max}}<\frac{2\eta}{n(m-1)}\quad\Longrightarrow\quad \|nx\|<\frac{2\eta}{m-1}.
	\end{equation}
	In summary, we have proved
	\[f_Q(x)=m\quad\Longrightarrow\quad \|nx\|<\frac{2\eta}{m-1} \text{ for some $n\le \frac{Q}{m-1}$}.\]
	This implies that
	\begin{align}
		\int_{\{f_Q\in[m,2m)\}}f_Q\dif\mu_\omega&\le
		2m\int\cha_{\bigcup_{q=1}^{Q/(m-1)}A(q,2\eta/(m-1),0)}\dif\mu_\omega\notag\\
		&\le 2m \int \sum_{q=1}^{Q/(m-1)}\cha_{A(q,2\eta/(m-1),0)}\dif\mu_\omega.\label{eq:fQ<m}
	\end{align}

	Define a new function
	\begin{equation}\label{eq:g}
		g=\sum_{q=1}^{Q/(m-1)}\cha_{A(q,2\eta/(m-1),0)}.
	\end{equation}
	We can modify $g$ by replacing the charater functions with smooth functions just as in the proof of \cite[Theorem 4.1]{Yufractals}, and we see that it is possible to show that
	\begin{equation}\label{eq:Fourier g}
		\int_{[0,1]} g\dif\mu_\omega\ll\int_{[0,1]}g(x)\dif x+\sum_{\xi\in\Z:0<|\xi|\le Q/\eta}|\widehat{\mu_\omega}(\xi)\widehat  g(\xi)|.
	\end{equation}
	The first term on the right can be easily computed and we have
	\begin{equation}\label{eq:volume g}
		\int_{[0,1]}g(x)\dif x\stackrel{\eqref{eq:g}}{\ll}\frac{Q}{m-1}\cdot \frac{2\eta}{m-1}\asymp\frac{Q\eta}{m^2}.
	\end{equation}
	For the second term, observe that for $\xi\ne 0$, we have
	\[|\widehat g(\xi)|\ll \sum_{q=1}^{Q/(m-1)}\eta/m \cdot \cha_{q|\xi}\le d(\xi)\eta/m,\]
	where $d(\xi)$ denote the number of divisors of $\xi$. The proof of this inequality can be found in \cite[page 16]{Yufractals}.
	This together with \eqref{eq:branch Fourier} gives
	\[\begin{split}
		\sum_{\xi\in\Z:0<|\xi|\le Q/\eta}|\widehat{\mu_\omega}(\xi)\widehat  g(\xi)|&\ll \sum_{\xi\in\Z:0<|\xi|\le Q/\eta}|\widehat{\mu_\omega}(\xi)| d(\xi)\eta/m\\
		&=\eta/m \cdot \sum_{\xi\in\Z:0<|\xi|\le Q/\eta}|\widehat{\mu}(b^{-|\omega|}\xi)| d(\xi).
	\end{split}\]
	It then follows from an argument in \cite[Page 40]{Yufractals} that for each $1/2<\gamma<\fldim\mu,$
	\begin{equation}\label{eq:Fourier sum}
		\sum_{0<|\xi|\le Q/\eta}|\widehat{\mu}(b^{-|\omega|}\xi)| d(\xi)\ll b^{\gamma|\omega|}Q^{1-\gamma}\eta^{-(1-\gamma)}.
	\end{equation}
	Therefore, for
	\begin{equation}\label{eq:taubeta}
		1<\alpha<\frac{\gamma}{1-\gamma} \qaq 0<\beta<\frac{1-(1+\alpha)(1-\gamma)}{\gamma},
	\end{equation}
	 we have that
	\begin{equation}\label{eq:rho}
		b^{\gamma|\omega|}Q^{1-\gamma}\eta^{-(1-\gamma)}\le Q^{\rho}\quad\text{with $\rho=\gamma\beta+(1+\alpha)(1-\gamma)\in(0,1)$}
	\end{equation}
	whenever $Q$ is large enough (depending on $\alpha$ and $\beta$ only), $Q^{-\alpha}\le\eta\le Q^{-1}$ and  $|K_\omega|=b^{-|\omega|}\ge Q^{-\beta}$. This means that
	\begin{equation}\label{eq:fQ good}
		\int_{\{f_Q\in [m,2m)\}}f_Q\dif\mu_\omega\stackrel{\eqref{eq:fQ<m}}{\ll} m\int g\dif\mu_\omega\stackrel{\eqref{eq:Fourier g},\eqref{eq:volume g}, \eqref{eq:rho}}{\ll} \frac{Q\eta}{m}+Q^{\rho}\eta.
	\end{equation}

	Note that
	\[
	\int_{\{f_Q\ge k\}} f_Q \, \dif \mu_\omega
	=
	\sum_{j\ge 0:\, 2^j k \le Q}
	\int_{\{f_Q \in [2^j k, 2^{j+1} k)\}} f_Q \, \dif \mu_\omega.
	\]
	Since there are at most $\log_2 Q$ possibilities for $j$, it follows from \eqref{eq:fQ good} that
	\[
	\begin{split}
		\sum_{j\ge 0:\, 2^j k \le Q^{1-\rho'}}
		\int_{\{f_Q \in [2^j k, 2^{j+1} k)\}} f_Q \, \dif \mu_\omega
		&\ll
		\sum_{j\ge 0:\, 2^j k \le Q} \biggl(\frac{Q\eta}{2^j k} + Q^\rho \eta \biggr)\\
		&\ll
		\frac{Q \eta}{k} + Q^\rho \eta \log_2 Q,
	\end{split}
	\]
	which completes the proof of the claim.
\end{proof}
\begin{remark}\label{r:alpha< missing digits}
The choice of the parameters $\alpha$ and $\beta$ in Lemma \ref{l:local inhomo} indeed satisfies
\[
\alpha \le \frac{1+\delta-\beta\delta}{1-\delta},
\]
as stated in Remark \ref{r:alphabeta}. To prove this inequality, it suffices to show that
\[
1-\alpha + (1+\alpha-\beta)\delta \ge 0.
\]
	Using the upper bound $\beta<\frac{1-(1 +\alpha)(1-\gamma)}{\gamma}$, we obtain
	\[\begin{split}
		1-\alpha+(1+\alpha-\beta)\delta&\ge 1-\alpha+\bigg(1+\alpha-\frac{1-(1 +\alpha)(1-\gamma)}{\gamma}\bigg)\delta\\
		&=1+\alpha(\delta/\gamma-1).
	\end{split}\]
	The desired conclusion then follows from the fact that, for missing digits measures,
	$\delta \ge \fldim \mu > \gamma$ (see \cite[Lemma 1.4 (1) and \S 3.6]{Yufractals}).
\end{remark}
\begin{proof}[Proof of Theorem \ref{t:example2}]
	(1) Since $\fldim\mu>1/2$, it follows from Lemma~\ref{l:local inhomo} that $\mu$ satisfies the $(\alpha,\beta,\bzero)$-local counting property for every $\theta\in\R$. The conclusion follows directly from Corollary \ref{c:corollary} (1).

	(2) Since $\hdim K\fldim\mu=\delta\cdot \fldim\mu>1/2$, we can choose $\gamma$ such that $1/2 < \gamma < \fldim\mu$ and $\delta \gamma > 1/2$.  With this choice of $\gamma$, we further choose parameters $\alpha$ and $\beta$ such that
	\begin{equation}\label{eq:alpha additional missing}
		1<\alpha<\frac{2\delta \gamma-\gamma}{1-\gamma}\qaq 2(1-\delta)\le\beta<\frac{1-(1+\alpha)(1-\gamma)}{\gamma}.
	\end{equation}
	 In order for the set of admissible $\beta$ values to be nonempty, and thus for $\delta \ge 1 - \beta/2$ to hold, the upper bound that $\alpha$ can take must be less than $\frac{\gamma}{1-\gamma}$. Indeed, the set of admissible $\beta$ values is nonempty if and only if
	 \[
	 2(1-\delta)<\frac{1-(1+\alpha)(1-\gamma)}{\gamma},
	 \]
	 which is exactly equivalent to
	 \[
	 \alpha<\frac{2\delta \gamma-\gamma}{1-\gamma},
	 \] as stated in \eqref{eq:alpha additional missing}. Once $\beta$ is chosen according to the lower bound in \eqref{eq:alpha additional missing}, the inequality $\delta \ge 1-\beta/2$ is automatically satisfied.

	  For the parameters $\alpha$ and $\beta$ chosen above, it follows from Lemma~\ref{l:local inhomo} that $\mu$ satisfies the $(\alpha,\beta,0)$-local counting property. Since $\delta\ge 1-\beta/2$, by Corollary \ref{c:corollary} (2), we have
		\[	\hm^{\delta+\frac{2}{1+\tau}-1}\big(K\cap E_1(\tau,0)\big)=\infty,\]
		whenever $1<\tau<\alpha$.
\end{proof}
\section{Proof of Theorem \ref{t:example3}}\label{s:higher dimension}
Throughout this section, let $K\subset\R^d$ be a missing digits set equipped with a $\delta$-Ahlfors regular measure $\mu$. Assume further that the following conditions hold:
\begin{enumerate}[({A}1)]
	\item $K = K_1 \times \cdots \times K_{d-1} \times [0,1]$, where each $K_j$ ($1\le j\le d-1$) is a missing digits set associated with the digit set $\mathcal{D}_j \subset \{0,1,\dots,b-1\}$;
	\item there exists $\frac{d}{d+1} < \gamma < 1$ such that each missing digits measure $\mu_j$ ($1 \le j \le d-1$) associated with $K_j$ satisfies $\fldim \mu_j > \gamma$.
\end{enumerate}
For convenience, we set $K_d = [0,1]$ and denote the corresponding missing digits measure by $\mu_d$. Clearly, the digit set $\cd_d$ associated with $K_d$ is $\{0,1,\dots,b-1\}$; moreover, we have $\fldim \mu_d =1> \gamma$.
 For each $1 \le j \le d$, write $\hdim K_j = \delta_j$. Then $\delta=\sum_{j=1}^d\delta_j=1+\sum_{j=1}^{d-1}\delta_j$.

The upper bound for $K\cap W_d(\psi_\tau,\bzero)$ is a direct consequence of \eqref{eq:upper bound} in Theorem \ref{t:main} and Corollary \ref{c:self-similar set}, since the missing digits set $K$ satisfying (A1) is strongly irreducible and satisfies the OSC.
\begin{lemma}\label{l:upper bound missing}
	There exists $\alpha>1/d$ such that
	\[\hdim \big(K\cap W_d(\psi_\tau,\bzero)\big)\le \delta+\frac{1+d}{1+\tau}-d,\quad\text{for any $\tau\in(1/d,\alpha)$}.\]
\end{lemma}
\begin{proof}
	The product structure of $K$ (see (A1)) guarantees that it is strongly irreducible, while its missing digits nature ensures that $K$ satisfies the OSC.
	By Corollary~\ref{c:self-similar set}, $\mu$ satisfies the $(\alpha,\beta,\bzero)$-local counting property for some $\alpha > 1/d$ and $\beta > 0$. The conclusion of the lemma then follows directly from \eqref{eq:upper bound} in Theorem~\ref{t:main}.
\end{proof}

In order to prove the lower bound for $\hdim(K \cap W_d(\psi_\tau, \mathbf{0}))$, we first establish several auxiliary results. For any $q\in\N\setminus\{0\}$ and $\bmeta=(\eta_1,\dots,\eta_d)\in (\R^+)^d$, define
\[R(q,\bmeta):=\{(x_1,\dots,x_d)\in[0,1]^d:\|qx_j\|\le\eta_j\text{ for $1\le j\le d$}\}.\]
\begin{lemma}[{\cite[Lemma 5.2]{HeWeightedMA}}]\label{l:meaupp}
	Let $q\in\N\setminus\{0\}$ and $\bmeta\in(0,1/2)^d$. Let $\nu$ be a Borel probability measure supported on $[0,1]^d$. Then,
	\[\nu\big(R(q,\bmeta)\big)\ll\eta_1\cdots\eta_d\Bigg(1+\sum_{\substack{\bxi\in\Z^d\setminus\{\bzero\}\\ \forall1\le j\le d, |\xi_j|\le 2/\eta_j}}|\widehat{\nu}(q\bxi)|\Bigg),\]
	where the implied constant is independent of $\nu$.
\end{lemma}
The following lemma can be viewed as a variant of the quantitative non-divergence result of Bernik, Kleinbock, and Margulis \cite[Theorem 1.4]{BerKleMarIMRN}, which originally concerned the Lebesgue measure on manifolds, and is now adapted to missing digits measures. It should be noted that the proof is similar to that of \eqref{eq:fQ good}.
\begin{lemma}\label{l:nondivergence}
Let $\frac{d}{d+1}<\gamma<1$ be as in assumption (A2), and let
\[	1\le \alpha<\frac{1-d+d\gamma}{1-\gamma}\qaq 0<\beta<\frac{1-d+d\gamma+\alpha(\gamma-1)}{d\gamma}.\]

	Let $\boa=(\omega_1,\dots,\omega_d)\in\cd^*:=(\prod_{j=1}^{d}\cd_j)^*$. Then, we have
	\begin{equation}\label{eq:non-divergence estimate}
		\sum_{1\le q\le Q}\mu_\boa\big(R(q,\bmeta)\big)\ll Q\eta_1\cdots\eta_d
	\end{equation}
	whenever $Q$ is large enough (depending only on $\alpha$ and $\beta$), $|K_\boa|=b^{-|\boa|}\ge Q^{-\beta}$ and  $\bmeta=(\eta_1,\dots,\eta_d)\in (0,1/2)^d$ with $Q^{-\alpha}\le \eta_1\cdots\eta_d<2^{-d} $. Consequently, if $B\subset K$ satisfies $|B|\ge Q^{-\beta}$, then
	\begin{equation}\label{eq:non-divergence estimate for B}
		\sum_{1\le q\le Q}\mu\big(B\cap R(q,\bmeta)\big)\ll \mu(B)\cdot Q\eta_1\cdots\eta_d.
	\end{equation}
\end{lemma}
\begin{proof}
	Note that the set of admissible choices for the parameters $\alpha$ and $\beta$ is nonempty, since $\frac{d}{d+1} < \gamma < 1$. Moreover, these choices imply
		\begin{equation}\label{eq:ab}
		d\gamma\beta+d-d\gamma+\alpha(1-\gamma)<1.
	\end{equation}

	  Write $\eta=\eta_1\cdots\eta_d$. By Lemma \ref{l:meaupp}, we have
	\[\begin{split}
		\sum_{1\le q\le Q}\mu_\boa\big(R(q,\bmeta)\big)&\ll Q\eta+\eta\cdot\sum_{\substack{\bxi\in\Z^d\setminus\{\bzero\}\\ \forall1\le j\le d, |\xi_j|\le 2/\eta_j}}\sum_{1\le q\le Q}|\widehat{\mu_\boa}(q\bxi)|.
	\end{split}\]
	Observe that for any $\bp\in\Z^d$, there are at most $d(|\bp|)$ many pairs $(\bxi,q)$ such that $q\bxi=\bp$, where $d(|\bp|)$ is the number of divisors of $|\bp|$, with $|\cdot|$ denoting the supremum norm, as before. Therefore,
	\[\sum_{\substack{\bxi\in\Z^d\setminus\{\bzero\}\\ \forall1\le j\le d, |\xi_j|\le 2/\eta_j}}\sum_{1\le q\le Q}|\widehat{\mu_\boa}(q\bxi)|\le \sum_{\substack{\bxi\in\Z^d\setminus\{\bzero\}\\ \forall1\le j\le d, |\xi_j|\le 2Q/\eta_j}}d(|\bxi|)\cdot|\widehat{\mu_\boa}(\bxi)|.\]
	By the assumption (A1), we have $\mu=\mu_1\times\cdots\times \mu_d$. Note that $d(|\bxi|)\le d(|\xi_1|)\cdots d(|\xi_d|)$, with the convention $d(0)=1$ so that the inequality remains valid when some $\xi_j=0$.
	 It follows that
	\[\begin{split}
		\sum_{\substack{\bxi\in\Z^d\setminus\{\bzero\}\\ \forall1\le j\le d, |\xi_j|\le 2Q/\eta_j}}d(|\bxi|)\cdot |\widehat{\mu_\boa}(\bxi)|&=\sum_{\substack{\bxi\in\Z^d\setminus\{\bzero\}\\ \forall1\le j\le d, |\xi_j|\le 2Q/\eta_j}}d(|\bxi|)\cdot |\widehat{\mu_{1,\omega_1}}(\xi_1)|\cdots |\widehat{\mu_{d,\omega_d}}(\xi_d)|\\
		&\le \prod_{j=1}^{d}\sum_{\substack{\xi\in\Z:|\xi|\le 2Q/\eta_j}}d(\xi)\cdot|\widehat{\mu_{j,\omega_j}}(\xi)|\\
		&\stackrel{\eqref{eq:branch Fourier}}{=}\prod_{j=1}^{d}\sum_{\substack{\xi\in\Z:|\xi|\le 2Q/\eta_j}}d(\xi)\cdot|\widehat{\mu_{j}}(b^{-k}\xi)|.
	\end{split}\]
	By the assumptions (A1) and (A2),  for $1\le j\le d$ we have $\fldim\mu_j>\gamma>\frac{d}{d+1}>\frac{1}{2}$. Applying \eqref{eq:Fourier sum}, we obtain
	\[\begin{split}
		\prod_{j=1}^{d}\sum_{\xi\in\Z:|\xi|\le 2Q/\eta_j}d(\xi)\cdot|\widehat{\mu_{j}}(b^{-k}\xi)|&\ll \prod_{j=1}^d b^{\gamma k}Q^{1-\gamma}\eta_j^{-(1-\gamma)}=b^{d\gamma k}Q^{d-d\gamma}\eta^{-(1-\gamma)}\\
		&\le Q^{d\gamma\beta} \cdot Q^{d-d\gamma}\cdot Q^{\alpha(1-\gamma)}\\
		&\stackrel{\eqref{eq:ab}}{<}Q.
	\end{split}\]
Combining these estimates, we obtain the first assertion of the lemma, namely \eqref{eq:non-divergence estimate}. Moreover, since $K$ satisfies the OSC, a similar argument as in the proof of Corollary \ref{c:self-similar set} shows that \eqref{eq:non-divergence estimate for B} also holds.
\end{proof}
We remark that, unlike the local counting property, we do not require that $\eta_1 \cdots \eta_d \le Q^{-1}$ in Lemma \ref{l:nondivergence}. Using the quantitative non-divergence estimate for missing digits measures (Lemma~\ref{l:nondivergence}), we establish the following result, which in the literature is referred to as {\em local ubiquity for rectangles} \cite{KleinbockWangadv,WangWuMTPrectangleMA}.
\begin{lemma}\label{l:ubiquitous}
	There exists a constant $c>0$, depending only on the implied constant in Lemma~\ref{l:nondivergence}, such that for any ball $B\subset K$, any $Q\in\N$ with $Q^{\beta}\ge |B|^{-1}$ (where $\beta$ is as in Lemma~\ref{l:nondivergence}), and any $\bmeta=(\eta_1,\dots,\eta_d)\in(0,1/2)^d$ satisfying $\eta_1\cdots\eta_d=Q^{-1}$, we have
	\[\mu\bigg(B\cap \bigcup_{q=cQ}^{2Q}R(q,\bmeta)\bigg)\asymp \mu(B).\]
	\end{lemma}
\begin{proof}
	We prove only the lower bound in the lemma, since the upper bound is immediate.

	 Since $\eta_1\cdots\eta_d=Q^{-1}$, by Minkowski's linear forms theorem in the geometry of numbers, for any $\bx\in B$ and $Q\in\N\setminus\{0\}$ there is a solution $(q,\bp)\in\Z^{d+1}\setminus\{\bzero\}$ to the system
	\[\begin{cases}
		|qx_1-p_1|<\eta_1\\
		\quad\vdots\\
		|qx_d-p_d|<\eta_d\\
		|q|\le 2Q.
	\end{cases}\]
Note that if $(q,\mathbf{p})$ is a solution, then so is $-(q,\mathbf{p})$. It then follows that
\[\mu\bigg(B\cap \bigcup_{q=1}^{2Q}R(q,\bmeta)\bigg)= \mu(B).\]
On the other hand, the assumptions of the lemma ensure that Lemma~\ref{l:nondivergence} is applicable. By Lemma \ref{l:nondivergence}, for any $c>0$ we have
\[\mu\bigg(B\cap \bigcup_{q=1}^{cQ}R(q,\bmeta)\bigg)\ll \mu(B)\cdot cQ\eta_1\cdots\eta_d=c\mu(B).\]
Choosing $0<c<2$ sufficiently small (depending only on the above implied constant), we may ensure that
\[\mu\bigg(B\cap \bigcup_{q=1}^{cQ}R(q,\bmeta)\bigg)\le \mu(B)/2.\]
Combining the above estimates, we obtain
\[\mu\bigg(B\cap \bigcup_{q=cQ}^{2Q}R(q,\bmeta)\bigg)\ge \mu\bigg(B\cap \bigcup_{q=1}^{2Q}R(q,\bmeta)\bigg)-\mu\bigg(B\cap \bigcup_{q=cQ}^{2Q}R(q,\bmeta)\bigg)\ge\mu(B)/2,\]
which establishes the desired lower bound.
\end{proof}
In order to extract a disjoint subfamily from a collection of rectangles whose suitable enlargements cover the original family, we shall rely on the following rectangular version of the classical $5r$-covering lemma.

\begin{lemma}[$5r$-covering lemma for rectangles, {\cite[Lemma 5.1]{LiLiaoVelaniWangZorinMTPadv}}]\label{l:5r covering}
	For each $1\le j\le d$, let $K_j$ be a locally compact subset of $\R$ and let $\{x_{j,n}\}_{n\ge 1}$ be a sequence of points in $K_j$. For any $r>0$ and any  $(u_1,\dots,u_d)\in (\R^+)^d$, the family
	\[\cg:=\bigg\{\prod_{j=1}^{d}B(x_{j,n},r^{u_j}):n\ge 1\bigg\}\] of rectangles admits a subfamily $\cg_1$ of disjoint rectangles satisfying
	\[\bigcup_{R\in\cg}R\subset\bigcup_{R\in\cg_1} 5^{\frac{\max_{1\le j\le d}u_j}{\min_{1\le j\le d}u_j}}R,\]
	where $tR$ denotes the rectangle obtained by enlarging the side lengths of $R$ by a factor of $t$, while keeping its center fixed.
\end{lemma}
We are now ready to prove the lower bound for $\hdim(K\cap W_d(\psi_\tau,\bzero))$.

\begin{lemma}
	There exists $1/d<\alpha<1/(d-1)$ such that
	\[\hm^{\delta+\frac{1+d}{1+\tau}-d}\big(K\cap E_d(\psi_\tau,\bzero)\big)=\infty,\quad\text{for any $\tau\in(1/d,\alpha)$}.\]
\end{lemma}

\begin{proof}
By the upper bound established in Lemma~\ref{l:upper bound missing}, together with the argument used in the proof of Corollary~\ref{c:corollary} (2) (see Section~\ref{s:corollary}), it suffices to show that
	\begin{equation}\label{eq:full measure}
		\hm^{\delta+\frac{1+d}{1+\tau}-d}\big(K\cap W_d(\psi_\tau,\bzero)\big)=\infty.
	\end{equation}
	Recall from \eqref{eq:A(q,eta,theta)} that
	$A(q,\eta,\bzero)=\{\bx\in\R^d:\|q\bx\|<\eta\}$.
	Since
	\[\begin{split}
		\bx\in W_d(\psi_\tau,\bzero)\quad&\Longleftrightarrow\quad \bx\in A(q,q^{-\tau},\bzero)\text{ for infinitely many $q\in\N$}\\
		&\Longleftrightarrow\quad \bx\in \bigcup_{q=Q}^{\infty}A(q,q^{-\tau},\bzero)\text{ for infinitely many $Q\in\N$},
	\end{split}\]
	we have
	\[W_d(\psi_\tau,\bzero)=\limsup_{Q\to\infty}\bigcup_{q=Q}^{\infty}A(q,q^{-\tau},\bzero)\]
	By Theorem \ref{t:weaken} (2), \eqref{eq:full measure} follows if
	\begin{equation}\label{eq:content bound missing}
		\limsup_{Q\to\infty}\hc^{\delta+\frac{1+d}{1+\tau}-d}\bigg(B\cap \bigcup_{q=Q}^{\infty}A(q,q^{-\tau},\bzero)\bigg)\gg \mu(B)
	\end{equation}
	holds for all ball $B$ in $K$.

	 Since $1/d<\tau<\alpha<1/(d-1)$, we have
	 \begin{equation}\label{eq:1-(d-1)tau<tau}
	 	0<1-(d-1)\tau<1/d<\tau.
	 \end{equation}
	 Let $Q$ be sufficiently large so that Lemma~\ref{l:ubiquitous} is applicable. For
	 \[\bmeta=(Q^{-\tau}/4,\dots,Q^{-\tau}/4, 4^{d-1}Q^{-1+(d-1)\tau}),\]
	 we have $\eta_1\cdots\eta_d=Q^{-1}$ and so by Lemma \ref{l:ubiquitous},
	\begin{equation}\label{eq:compare mu(B)}
		\mu\bigg(B\cap \bigcup_{q=cQ}^{2Q}R(q,\bmeta)\bigg)\asymp \mu(B),
	\end{equation}
	where $c>0$ is as in Lemma \ref{l:ubiquitous}.
	Let $\cq=\cq(B)$ be the collection of rational vectors $\bp/q\in\Q^d$ such that \[B\cap \left(\Bigg(\prod_{j=1}^{d-1}B\bigg(\frac{p_j}{q}, \frac{Q^{-\tau}}{4q}\bigg)\Bigg)\times B\bigg(\frac{p_d}{q},\frac{4^{d-1}Q^{-1+(d-1)\tau}}{q}\bigg)\right)\ne\varnothing.\]
	Although the above intersection may be very small, doubling the side lengths of the rectangles ensures that the resulting intersection is sufficiently large for our purposes. More precisely, the intersection
	\begin{equation}\label{eq:intersection}
		B\cap \left(\Bigg(\prod_{j=1}^{d-1}B\bigg(\frac{p_j}{q}, \frac{Q^{-\tau}}{2q}\bigg)\Bigg)\times B\bigg(\frac{p_d}{q},\frac{4^{d}Q^{-1+(d-1)\tau}}{2q}\bigg)\right)
	\end{equation}
	contains a rectangle of the form
	\begin{equation}\label{eq:contain a rectangle}
		\Bigg(\prod_{j=1}^{d-1}B\bigg(x_j, \frac{Q^{-\tau}}{4q}\bigg)\Bigg)\times B\bigg(\frac{p_d}{q},\frac{4^{d-1}Q^{-1+(d-1)\tau}}{q}\bigg)
	\end{equation}
	for some $\bx=(x_1,\dots,x_{d-1},p_d/q)\in B$. Here, the choice of the last coordinate of $\bx$ as $p_d/q$ is possible since $K_d=[0,1]$. Let $\cg$ denote the collection of all rectangles in \eqref{eq:contain a rectangle} obtained by ranging over $\bp/q\in\cq$. By the $5r$-covering lemma for rectangles (Lemma \ref{l:5r covering}), there exists a subfamily $\cg_1\subset\cg$ consisting of disjoint rectangles such that
	\[\bigcup_{R\in\cg}R\subset\bigcup_{R\in\cg_1} 5^{\frac{1+\tau}{(d-1)\tau}}R.\]
	By the product structure of $\mu$ (that is, $\mu=\mu_1\times\cdots\times\mu_d$) and the $\delta_j$-Ahlfors regularity of each factor measure $\mu_j$, we have $\mu(R)\asymp\mu(5^{\frac{1+\tau}{(d-1)\tau}}R)$. The disjointness of the rectangles in  $\cg_1$
	then allows us to estimate the measure of their union
	\[\begin{split}
		\mu\bigg(\bigcup_{R\in\cg_1}R\bigg)&=\sum_{R\in\cg_1}\mu(R)\asymp\sum_{R\in\cg_1}\mu\big(5^{\frac{1+\tau}{(d-1)\tau}}R\big)\ge \mu\bigg(\bigcup_{R\in\cg_1} 5^{\frac{1+\tau}{(d-1)\tau}}R\bigg)\\
		&\ge \mu\bigg(\bigcup_{R\in\cg} R\bigg)\stackrel{\eqref{eq:compare mu(B)}}{\gg}\mu(B).
	\end{split}\]
	Let $\cx$ denote the collection of centers of the rectangles in $\mathcal{G}_1$. By the disjointness of these rectangles and a simple volume argument,
	\begin{equation}\label{eq:number}
		\#\cx\asymp \frac{\mu(B)}{Q^{-2+(d-1)\tau}\cdot \prod_{j=1}^{d-1}Q^{-(1+\tau)\delta_j}}=\mu(B)\cdot Q^{(1+\tau)(\delta-1)+2-(d-1)\tau},
	\end{equation}
	where the equality follows from $\delta_1+\cdots+\delta_d=\delta$ and $\delta_d=1$.

	Note from \eqref{eq:1-(d-1)tau<tau} that $1-(d-1)\tau<\tau$. We shrink each rectangle in $\mathcal{G}_1$ in the $d$-th direction to obtain a family of well-separated balls, defined by
	\[\bigcup_{\bx\in\cx}B\bigg(\bx, \frac{Q^{-\tau}}{4q}\bigg)=:F.\]
	Denote the collection of these balls by $\cb$.
	Since the centres of the balls in $\mathcal{B}$ are of the form $(x_1,\dots,x_{d-1},p_d/q)\in B$ (see \eqref{eq:contain a rectangle}) and $K_d=[0,1]$, it follows from \eqref{eq:contain a rectangle} that for any $\bx\in\cx$,
	\[B\cap B\bigg(\bx, \frac{Q^{-\tau}}{4q}\bigg)\subset B\cap B\bigg(\frac{\bp}{q}, \frac{Q^{-\tau}}{2q}\bigg).\]
	Consequently,
	\[F\subset B\cap \bigcup_{q=cQ}^{2Q}A(q,q^{-\tau},\bzero).\]

	Define a probability measure $\nu$ supported on $F$ by
	\begin{equation}\label{eq:nu2}
		\nu=\sum_{D\in\cb}\frac{1}{\#\cb}\cdot\frac{\mu|_D}{\mu(D)}.
	\end{equation}
	Clearly, we have
	\begin{equation}\label{eq:number cb}
		\#\cb=\#\cx\asymp\mu(B)\cdot Q^{(1+\tau)(\delta-1)+2-(d-1)\tau}.
	\end{equation}

	Next, we estimate the $\nu$-measure of an arbitrary ball. Let $\bx \in F$ and let $0<r<|B|$. We proceed by considering several ranges of $r>0$.

	\noindent {\bf Case 1}: $Q^{-(1+\tau)}\le r<|B|$. Since the collection $\mathcal{B}$ of balls is obtained by shrinking the rectangles in $\mathcal{G}_1$ along the $d$-th coordinate, and since $r\ge Q^{-(1+\tau)}$ exceeds the minimal side lengths of these rectangles, it follows that the ball $B(\bx,r)$ can intersect at most
\begin{align}
	&\asymp\frac{r^{\delta_1+\cdots+\delta_{d-1}}}{Q^{-(1+\tau)(\delta_1+\dots+\delta_{d-1})}}\cdot \bigg(\frac{r}{Q^{-2+(d-1)\tau}}+2\bigg)\notag\\
	&\asymp r^{\delta-1} Q^{(1+\tau)(\delta-1)}\cdot\max \big\{rQ^{2-(d-1)\tau},2\big\}\label{eq:max}
\end{align}
	rectangles in $\cg_1$.  Consequently, this also provides an upper bound on the number of balls in $\cb$ that can intersect $B(\bx,r)$.

	If the maximum in \eqref{eq:max} is $rQ^{2-(d-1)\tau}$, then we have
	\[\nu\big(B(\bx,r)\big)\stackrel{\eqref{eq:nu2}, \eqref{eq:number cb}}{\ll} \frac{r^\delta Q^{(1+\tau)(\delta-1)+2-(d-1)\tau}}{\mu(B)\cdot Q^{(1+\tau)(\delta-1)+2-(d-1)\tau}}=\frac{r^\delta}{\mu(B)}\ll\frac{r^{\delta+\frac{1+d}{1+\tau}-d}}{\mu(B)},\]
	since $\tau>1/d$.

	Otherwise, suppose that the maximum in \eqref{eq:max} is $2$. Then, we have
	\[\begin{split}
		\nu\big(B(\bx,r)\big)&\stackrel{\eqref{eq:nu2}, \eqref{eq:number cb}}{\ll} \frac{2r^{\delta-1} Q^{(1+\tau)(\delta-1)}}{\mu(B)\cdot Q^{(1+\tau)(\delta-1)+2-(d-1)\tau}}\asymp\frac{r^{\delta-1}Q^{-(2-(d-1)\tau)}}{\mu(B)}.
	\end{split}\]
	Since $\tau<1/(d-1)$, we have $2-(d-1)\tau>1>0$. By $Q^{-(1+\tau)}\le r$, we have
	\[Q^{-(2-(d-1)\tau)}=\big(Q^{-(1+\tau)}\big)^{\frac{2-(d-1)\tau}{1+\tau}}\le r^{\frac{2-(d-1)\tau}{1+\tau}}.\]
	Therefore,
	\[\nu\big(B(\bx,r)\big)\ll \frac{r^{\delta-1+\frac{2-(d-1)\tau}{1+\tau}}}{\mu(B)}=\frac{r^{\delta+\frac{1+d}{1+\tau}-d}}{\mu(B)}.\]
	\noindent {\bf Case 2}: $0< r< Q^{-(1+\tau)}$.  For any ball $D\in\cb$, we have
	\[\begin{split}
		\nu\big(B(\bx,r)\cap D\big)&\ll \frac{1}{\#\cb}\cdot \frac{\mu\big(B(\bx,r)\big)}{\mu(D)}\stackrel{\eqref{eq:number cb}}{\ll} \frac{1}{\mu(B)\cdot Q^{(1+\tau)(\delta-1)+2-(d-1)\tau}}\cdot\frac{r^\delta}{Q^{-(1+\tau)\delta}}\\
		&=\frac{r^\delta Q^{d\tau-1}}{\mu(B)}\le \frac{r^{\delta+\frac{1-d\tau}{1+\tau}}}{\mu(B)}=\frac{r^{\delta+\frac{1+d}{1+\tau}-d}}{\mu(B)}.
	\end{split} \]
	Since the balls in $\cb$ are pairwise disjoint, the ball $B(\bx,r)$ can intersect at most $\asymp 1$ balls in $\cb$. Consequently, we obtain
	\[\nu\big(B(\bx,r)\big)\ll \frac{r^{\delta+\frac{1+d}{1+\tau}-d}}{\mu(B)}.\]

	By Cases 1 and 2 above, we have
\[\nu\big(B(\bx,r)\big)\ll \frac{r^{\delta+\frac{1+d}{1+\tau}-d}}{\mu(B)}.\]
Note that $\nu$ is supported on  $F\subset B\cap \bigcup_{q=cQ}^{2Q}A(q,q^{-\tau},\bzero)$. Therefore, by mass distribution principle (see Proposition \ref{p:MDP}),
\[\hc^{\delta+\frac{1+d}{1+\tau}-d}\bigg(B\cap \bigcup_{q=cQ}^{2Q}A(q,q^{-\tau},\bzero)\bigg)\gg\mu(B).\]
Since this holds for any $Q$ sufficiently large $Q$, the lower bound in \eqref{eq:content bound missing} follows, which completes the proof the lemma.
\end{proof}

\subsection*{Acknowledgements}
Y. He was supported by the NSFC (No. 12401108) and partially by a grant from the Guangdong Provincial Department of Education (2025KCXT\\ D013).

\bibliographystyle{abbrv}

\bibliography{bibliography}

\end{document}